\documentclass[10pt]{amsart}

\usepackage[colorlinks]{hyperref}
\usepackage{color,graphicx,shortvrb}
\usepackage[latin 1]{inputenc}

\usepackage[active]{srcltx} %SRC Specials for DVI search

\usepackage{enumerate}
\usepackage{amssymb}

\newtheorem{theorem}{Theorem}[section]

\newtheorem{corollary}[theorem]{Corollary}

\newtheorem{lemma}[theorem]{Lemma}

\newtheorem{proposition}[theorem]{Proposition}
\newtheorem{remark}[theorem]{Remark}

\def\RR{{\mathbb{R}}}

\def\11{\textbf{$1$}}

\DeclareGraphicsExtensions{.jpg,.pdf,.png,.eps}

\begin{document}

\title[The Mazur-Ulam property for commutative von Neumann algebras]{The Mazur-Ulam property for commutative von Neumann algebras}

\author[A.M. Peralta, M. Cueto-Avellaneda]{Antonio M. Peralta, Mar{\'i}a Cueto-Avellaneda}

\address[A.M. Peralta, M. Cueto-Avellaneda]{Departamento de An{\'a}lisis Matem{\'a}tico, Facultad de
Ciencias, Universidad de Granada, 18071 Granada, Spain.}
\email{aperalta@ugr.es, mcueto@ugr.es}

%\thanks{}

\subjclass[2010]{Primary 47B49, Secondary 46A22, 46B20, 46B04, 46A16, 46E40.}

\keywords{Tingley's problem; Mazur-Ulam property; extension of isometries; $C(K)$; commutative von Neumann algebras, $L^{\infty} (\Omega,\mu)$.}

\date{}

\begin{abstract} Let $(\Omega,\mu)$ be a $\sigma$-finite measure space. Given a Banach space $X$, let the symbol $S(X)$ stand for the unit sphere of $X$. We prove that the space $L^{\infty} (\Omega,\mu)$ of all complex-valued measurable essentially bounded functions equipped with the essential supremum norm, satisfies the Mazur-Ulam property, that is, if $X$ is any complex Banach space, every surjective isometry $\Delta: S(L^{\infty} (\Omega,\mu))\to S(X)$ admits an extension to a surjective real linear isometry $T: L^{\infty} (\Omega,\mu)\to X$. This conclusion is derived from a more general statement which assures that every surjective isometry $\Delta : S(C(K))\to S(X),$ where $K$ is a Stonean space, admits an extension to a surjective real linear isometry from $C(K)$ onto $X$.
\end{abstract}

\maketitle
\thispagestyle{empty}

\section{Introduction}
A Banach space $X$ satisfies the \emph{Mazur-Ulam property} if for any Banach space $Y$, every surjective isometry $\Delta: S(X)\to S(Y)$ admits an extension to a surjective real linear isometry from $X$ onto $Y$, where $S(X)$ and $S(Y)$ denote the unit spheres of $X$ and $Y$, respectively. An equivalent reformulation tells that $X$ satisfies the Mazur-Ulam property if the so-called Tingley's problem admits a positive solution for every surjective isometry from $S(X)$ onto the unit sphere of any Banach space $Y$. Positive solutions to Tingley's problem have been established when $X$ and $Y$ are sequence spaces \cite{Ding2002,Di:p,Di:8,Di:1}, $L^{p}(\Omega, \Sigma, \mu)$ spaces with $1\leq p\leq \infty$ \cite{Ta:8, Ta:1, Ta:p}, $C(K)$ spaces \cite{Wang}, spaces of compact operators on complex Hilbert spaces and compact C$^*$-algebras \cite{PeTan16}, spaces of bounded linear operators on complex Hilbert spaces, atomic von Neumann algebras and JBW$^*$-triples \cite{FerPe17b,FerPe17c}, general von Neumann algebras \cite{FerPe17d}, spaces of trace class operators \cite{FerGarPeVill17}, preduals of von Neumann algebras \cite{Mori2017}, and spaces of $p$-Schatten von Neumann operators on a complex Hilbert space (with $2<p<\infty$) \cite{FerJorPer2018}.  We refer to the surveys \cite{Ding2009,YangZhao2014,Pe2018} for a detailed overview on Tingley's problem.\smallskip

Our knowledge on the class of Banach spaces satisfying the Mazur-Ulam property is a bit more reduced. This class includes the space $c_0(\Gamma,\mathbb{R})$ of real null sequences, and the space $\ell_{\infty}(\Gamma,\mathbb{R})$ of all bounded real-valued functions on a discrete set $\Gamma$ (see \cite[Corollary 2]{Ding07}, \cite[Main Theorem]{Liu2007}), the space $C(K,\mathbb{R})$ of all real-valued continuous functions on a compact Hausdorff space $K$ \cite[Corollary 6]{Liu2007}, and the spaces $L^{p}((\Omega, \mu), \mathbb{R})$ of real-valued measurable functions on an arbitrary $\sigma$-finite measure space $(\Omega, \mu)$ for all $1\leq p\leq \infty$ \cite{Ta:1,Ta:8,Ta:p}. For some time the study of those Banach spaces satisfying the Mazur-Ulam property was restricted to real Banach spaces. The existence of real linear surjective isometries which are not complex linear nor conjugate linear was a serious obstacle. Two recent contributions initiate the study of the Mazur-Ulam property in the setting of complex Banach spaces. Let $\Gamma$ be an infinite set, then the space of complex null sequences $c_0(\Gamma)$ satisfies the Mazur-Ulam property (see \cite{JVMorPeRa2017}). The space $\ell_{\infty}(\Gamma)$ of all complex-valued bounded functions on $\Gamma$ also satisfies the Mazur-Ulam property \cite{Pe2017}.\smallskip

In \cite{THL}, D. Tan, X. Huang, and R. Liu  introduce the notions of \emph{generalized lush} (GL) spaces and local-GL-spaces in the study of the Muzar-Ulam property by showing that every local-GL-space satisfies this property. Among the consequences of this, it is established that if $E$ is a local-GL-space and $K$ is a compact Hausdorff space, then $C(K,E)$ has the Mazur-Ulam property (see \cite[Proposition 3.11]{THL}). It should be observed that every CL-space in the sense of Fullerton \cite{Full61}, and every almost-CL-space in the sense employed by Lima in \cite{Lim77} is a GL-space. Let us briefly recall that a Banach space $X$ is a generalized lush space if for every $x\in S(X)$ and every $0<\varepsilon<1$ there exists a slice $S = S(\varphi,\varepsilon) = \{z\in X : \|z\|\leq 1, \Re\hbox{e}\varphi(z)>1-\varepsilon\}$ (with $\varphi \in S(X^*)$) such that $x\in S$ and $$\hbox{dist}(y,S) + \hbox{dist}(y,-S) < 2+\varepsilon,$$ for all $y\in S(X)$. It is not hard to check that $\mathbb{C}$ is not a (local-)GL-space. Therefore, the result established by Tan, Huang, and Liu in \cite[Proposition 3.11]{THL} does not throw any new light for the space $C(K)$ of all complex-valued functions on a compact Hausdorff space $K$.\smallskip

The space $L^{\infty}(\Omega, \mu)$ of complex-valued, measurable, essentially-bounded functions on an arbitrary $\sigma$-finite measure space $(\Omega, \mu)$ is beyond from our current knowledge on the class of complex Banach spaces satisfying the Mazur-Ulam property. This paper is devoted to fill this gap and clear our doubts.\smallskip

The natural path is to explore the interesting proof provided by D. Tan in the case of  $L^{\infty}(\Omega, \mu, \mathbb{R})$ in \cite{Ta:8}. A detailed checkup of the arguments in \cite{Ta:8} should convince the reader that those arguments are optimized for the real setting and it is hopeless to deal with complex scalars with the tools in \cite{Ta:8}. To avoid difficulties we extend our study to a wider setting of complex Banach spaces, the space of all complex-valued continuous functions on a Stonean space.\smallskip

Let $K$ be a compact Hausdorff space. We recall that $K$ is called \emph{Stonean} or \emph{extremally disconnected} if the closure of every open set in $K$ is open. It is known that if $K$ is a Stonean space, then every element $a$ in
the C$^*$-algebra $C(K),$ of all continuous complex-valued functions on $K$, can be uniformly approximated by finite linear combinations of projections (see \cite[Proposition 1.3.1]{S}). This topological notion has a straight connection with the property of being \emph{monotone complete}. More concretely, let $K$ be a compact Hausdorff space, then every bounded increasing directed set of real-valued non-negative functions $(f_{\alpha})$ in $C(K)$ has a least upper bound in $C(K)$ if and only if $K$ is Stonean (cf. \cite{Stone49} and \cite{Dix51} or \cite[Proposition 1.3.2]{S}, \cite[Proposition III.1.7]{Tak}). Let us mention, by the way, that a reader interested on a systematic comprehensive insight into the bewildering variety of monotone complete C$^*$-algebras beyond von Neumann algebras and commutative AW$^*$-algebras can consult the recent monograph \cite{SaiWri2015} by K. Sait\^{o} and J.M.D. Wright.\smallskip

The C$^*$-algebra $C(K)$ is a dual Banach space (equivalently, a von Neumann algebra) if and only if $K$ is hyper-Stonean (cf. \cite{Dix51}). We recall that a Stonean space $K$ is said to be hyper-Stonean if it admits a faithful family of positive normal measures (cf. \cite[Definition 1.14]{Tak}). \smallskip

Following standard terminology, a \emph{localizable measure space} $(\Omega, \nu)$ is a measure space which can be obtained as a direct sum of finite measure spaces $\{(\Omega_i,\mu_i) : i\in \mathcal{I}\}$. The Banach space $L^{\infty} (\Omega,\nu)$ of all essentially bounded locally $\nu$-measurable functions on $\Omega$ is a dual Banach space and a commutative von Neumann algebra. Actually, every commutative von Neumann algebra is C$^*$-isomorphic and isometric to some $L^{\infty} (\Omega,\nu)$ for some localizable measure space $(\Omega, \nu)$ (see \cite[Proposition 1.18.1]{S}). From the point of view of Functional Analysis, the commutative von Neumann algebras $L^{\infty} (\Omega,\nu)$ and $C(K)$ with $K$ hyper-Stonean are isometrically equivalent.\smallskip

In this paper we establish that if $K$ is a Stonean space, $X$ is an arbitrary complex Banach space, and $\Delta : S(C(K))\to S(X)$ is a surjective isometry, then there exist two disjoint clopen subsets $K_1$ and $K_2$ of $K$ such that $K = K_1 \cup K_2$
satisfying that if $K_1$ {\rm(}respectively, $K_2${\rm)} is non-empty then there exist a closed subspace $X_1$ {\rm(}respectively, $X_2${\rm)} of $X$ and a complex linear {\rm(}respectively, conjugate linear{\rm)} surjective isometry $T_1 : C(K_1) \to X_1$ {\rm(}respectively, $T_2 : C(K_2) \to X_2${\rm)} such that $X = X_1\oplus^{\infty} X_2$, and $\Delta (a) = T_1(\pi_1(a)) +  T_2(\pi_2(a))$ for every $a\in S(C(K))$, where $\pi_j$ is the natural projection of $C(K)$ onto $C(K_j)$ given by $\pi_j (a) = a|_{K_j}$. In particular, $\Delta$ admits an extension to a surjective real linear isometry from $C(K)$ onto $X$ (see Theorem \ref{t Tingley CK}).\smallskip

Let $(\Omega,\mu)$ be a $\sigma$-finite measure space, and let $X$ be a complex Banach space. A consequence of our main result shows that for every surjective isometry  $\Delta: S(L^{\infty} (\Omega,\mu))\to S(X)$, there exists a surjective real linear isometry $T: L^{\infty} (\Omega,\mu)\to X$ whose restriction to $S(L^{\infty} (\Omega,\mu))$ is $\Delta$ (see Theorem \ref{t Tingley Tan}).\smallskip

We finish this note with a discussion on the chances of extending a surjective isometry between the sets of extreme points of two Banach spaces.

\section{Geometric properties for general compact Hausdorff spaces}

In this section we shall gather a collection of results which are motivated by previous contributions in \cite{Wang,Ding07,FangWang06,Liu2007,Ta:8,JVMorPeRa2017} and \cite{Pe2017}.\smallskip

Henceforth, given a Banach space $X$, the symbol $\mathcal{B}_{X}$ will denote the closed unit ball of $X$.\smallskip

Let us consider a compact Hausdorff space $K$ and the C$^*$-algebra $C(K)$. For each $t_0\in K$ and each $\lambda\in \mathbb{T}$ we set $$A(t_0,\lambda):=\{ f\in S(C(K)) : f(t_0) = \lambda\},$$ where $\mathbb{T}$ denotes the unit sphere of $\mathbb{C}$. Then $A(t_0,\lambda)$ is a maximal norm-closed proper face of $\mathcal{B}_{C(K)}$ and a maximal convex subset of $S(C(K))$. As in previous papers, we consider a special subset of $A(t_0,\lambda)$ defined by $$\hbox{Pick}(t_0,\lambda) :=\{ f\in S(C(K)) : f(t_0) = \lambda, \hbox{ and } |f(t)|<1, \ \forall t\neq t_0\}.$$ It is known that in a compact metric space the set $\hbox{Pick}(t_0,\lambda)$ is non-empty for every $t_0\in K$. The same statement is actually true whenever $K$ is a first countable compact Hausdorff space (see \cite[proof of Theorem 2.2]{MolZa99}).\smallskip

Similar arguments to those given in \cite[Lemma 2.1]{JVMorPeRa2017} can be applied to establish our first result.

\begin{lemma}\label{l existence of support functionals for the image of a face} Let $\Delta : S(C(K))\to S(X)$ be a surjective isometry, where $K$ is a compact Hausdorff space and $X$ is a complex Banach space. Then for each $t_0\in K$ and each $\lambda\in \mathbb{T}$ the set $$\hbox{supp}(t_0,\lambda) := \{\varphi\in X^* : \|\varphi\|=1,\hbox{ and } \varphi^{-1} (\{1\})\cap \mathcal{B}_{X} = \Delta(A(t_0,\lambda)) \}$$ is a non-empty weak$^*$-closed face of $\mathcal{B}_{X^*}$.
\end{lemma}

\begin{proof} Since $A(t_0,\lambda)$ is a maximal convex subset of $S(C(K))$, we deduce from \cite[Lemma 5.1$(ii)$]{ChenDong2011} (see also \cite[Lemma 3.5]{Tan2014}) that $\Delta(A(t_0,\lambda)) $ is a maximal convex subset of $X$. Thus, by Eidelheit's separation Theorem \cite[Theorem 2.2.26]{Megg98} there is a norm-one functional $\varphi\in X^*$ such that $\varphi^{-1} (\{1\})\cap \mathcal{B}_{X} = \Delta(A(t_0,\lambda))$ (compare the proof of \cite[Lemma 3.3]{Tan2016}). The rest can be straightforwardly checked by the reader.
\end{proof}

Our next lemma was essentially shown in \cite[Lemmas 3.1 and 3.5]{FangWang06}, \cite[Lemma 2.4]{Ta:8} and \cite[Lemma 2.2]{JVMorPeRa2017}. We include an sketch of the proof for completeness. We recall first that given a norm-one element $x$ in a Banach space $X$, the \emph{star-like subset of $S(X)$ around $x$}, St$(x)$, is the set given by $$\hbox{St} (x) :=\{y\in S(X) : \|x+ y \|=2\}.$$ It is known that St$(x)$ is precisely the union of all maximal convex subsets of $S(X)$ containing $x,$ moreover, $$\hbox{St}(x)= \{y\in X : [x,y]=\{t x + (1-t) y : t\in [0,1]\}\subseteq S(X) \}.$$

\begin{lemma}\label{l 3.5 FangWang supports}{\rm\cite[Lemmas 3.1 and 3.5]{FangWang06}, \cite[Lemma 2.2]{JVMorPeRa2017}} Suppose $K$ is a first countable compact Hausdorff space, where $X$ is a complex Banach space. Let $\Delta : S(C(K))\to S(X)$ be a surjective isometry. Then for each $t_0$ in $K$ and each $\lambda\in \mathbb{T}$ we have $\varphi \Delta(f) = -1$, for every $f$ in $A(t_0,-\lambda)$ and every $\varphi \in \hbox{supp}(t_0,\lambda)$.
\end{lemma}

\begin{proof} Let us take $f\in A(t_0,-\lambda)$ and  $\varphi \in \hbox{supp}(t_0,\lambda)$. We can always pick $g_0$ in $\hbox{Pick}(t_0,\lambda)$ (here we need the hypothesis assuring that $K$ is a first countable compact Hausdorff space). Clearly $$\|\Delta(f)-\Delta(g_0) \| = \|f-g_0 \|=2,$$ and hence $- \Delta(f)\in \hbox{St} (\Delta(g_0)).$\smallskip

By mimicking the proof in \cite[Lemma 3.1]{FangWang06} we can show that $\hbox{St} (\Delta(g_0)) $ $= $ $\Delta (A(t_0,\lambda))$. Explicitly speaking, $z\in \hbox{St} (\Delta(g_0))$ if and only if $\| z + \Delta(g_0)\| = 2$. Applying \cite[Corollary 2.2]{FangWang06} we have $\| z + \Delta(g_0)\| = 2 \Leftrightarrow \|\Delta^{-1} (z) + g_0\|=2 \Leftrightarrow \Delta^{-1} (z) \in \hbox{St} (g_0)= A(t_0,\lambda).$ This shows that $- \Delta(f)\in \hbox{St} (\Delta(g_0))= \Delta (A(t_0,\lambda)),$ and hence $$- \varphi ( \Delta(f))= \varphi (- \Delta(f))= 1.$$
\end{proof}

We shall need an appropriate version of the above result in which $K$ is replaced with a compact Hausdorff space. We begin with a technical consequence of the parallelogram law.

\begin{lemma}\label{l parallelogram law} Let $\lambda_1,\lambda_2$ be two different numbers in $\mathbb{T}$. Then for every $0<\rho<\hbox{dist} (\lambda_1,[0,1]\lambda_2)$ we have $| \alpha + \beta| < \sqrt{4 - (\hbox{dist} (\lambda_1,[0,1]\lambda_2) - \rho)^2}<2,$ for every $\alpha\in \mathcal{B}_{\mathbb{C}}$ with $|\alpha -\lambda_1|<\rho$ and every $\beta \in [0,1]\lambda_2$.
\end{lemma}

\begin{proof} Let us denote $\theta = \hbox{dist} (\lambda_1,[0,1]\lambda_2) >0,$ and take any $0<\rho < \theta$. It is standard to check that $|\alpha -\beta| > \theta - \rho >0.$ By the parallelogram law we have $$ |\alpha +\beta|^2 +|\alpha -\beta|^2 = 2(|\alpha|^2+ |\beta|^2) \leq 4, $$ and thus $$ |\alpha +\beta|  \leq \sqrt{4 -  |\alpha -\beta|^2} < \sqrt{4 - (\theta - \rho)^2}<2.$$
\end{proof}

The extension of Lemma \ref{l 3.5 FangWang supports} for general compact Hausdorff spaces can be stated now.

\begin{lemma}\label{l 3.5 FangWang supports for general compact Hausdorff spaces} Suppose $K$ is a compact Hausdorff space. Let $\Delta : S(C(K))\to S(X)$ be a surjective isometry, where $X$ is a complex Banach space. Then for each $t_0$ in $K$ and each $\lambda\in \mathbb{T}$ we have $$\hbox{$\varphi \Delta(f) = -1,$ for every $f$ in $A(t_0,-\lambda)$ and every $\varphi \in \hbox{supp}(t_0,\lambda)$.}$$
Consequently, $\hbox{supp}(t_0,-\lambda) = -\hbox{supp}(t_0,\lambda),$ and $\Delta(-A(t_0,\lambda)) = -\Delta (A(t_0,\lambda))$.
\end{lemma}

\begin{proof} Let us take $f\in A(t_0,-\lambda)$ and  $\varphi \in \hbox{supp}(t_0,\lambda)$. The element $-\Delta (f)\in S(X)$, and thus there exists $h\in S(C(K))$ satisfying $\Delta (h) = -\Delta (f)$. We consider any $g\in A(t_0,\lambda)$. Since $\| f - g\| =2 = \|\Delta (f) - \Delta(g)\| =2,$ we deduce that $\Delta(h) = - \Delta(f)\in \hbox{St} (\Delta(g)).$ We have shown that $\|\Delta(h) + \Delta (g)\|=2,$ for all $g\in A(t_0,\lambda)$. Corollary 2.2 in \cite{FangWang06} implies \begin{equation}\label{eq 10 01}\hbox{ $\|h + g\|=2,$ for all $g\in A(t_0,\lambda)$. }
\end{equation} Consequently, for each $g \in A(t_0,\lambda)$ there exists $t_g\in K$ such that $$2\leq |h(t_g) + g(t_g)| \leq  |h(t_g)| + |g(t_g)|\leq 2.$$ That is, $|h(t_g)| =1$.\smallskip

For each open set $\mathcal{O}$ with $t_0\in \mathcal{O}$, we find, via Urysohn's lemma, $g_{_\mathcal{O}}\in A(t_0,\lambda)$ with $g_{_\mathcal{O}}|_{K\backslash\mathcal{O}} =0$. The above arguments show the existence of $t_{_\mathcal{O}}\in \mathcal{O}$ satisfying  $ |h(t_{_\mathcal{O}})| =1$ for every $\mathcal{O}.$ When the family of open subsets of $K$ containing $t_0$ are ordered by inclusion, the net $ (t_{_\mathcal{O}})_{_{\mathcal{O}}}$ converges to $t_0$. The continuity of $h$ gives $(1 )_{_{\mathcal{O}}}= (|h(t_{_\mathcal{O}})| )_{_{\mathcal{O}}}\to |h(t_{0})|$. Therefore, $|h(t_{0})|=1$.\smallskip

If $h(t_0) \neq \lambda$, we find, via Lemma \ref{l parallelogram law}, $0< \rho < \hbox{dist} (h(t_0),[0,1]\lambda)=\theta$ such that $| \alpha + \beta| \leq \sqrt{4 - (\theta - \rho)^2}<2$, for every $\alpha\in \mathcal{B}_{\mathbb{C}}$ with $|\alpha -h(t_{0})|<\rho$ and $\beta \in [0,1]\lambda$. The set $U:=\{s\in K : |h(s)-h(t_0)|<\rho\}$ is an open neighbourhood of $t_0$. Applying Urysohn's lemma we find $k\in C(K)$ with $0\leq k \leq 1$, $k(t_0) =1,$ and $k|_{_{K\backslash U}}=0$. The function $\lambda k\in A(t_0,\lambda)$, and then \eqref{eq 10 01} implies that $\|h + \lambda k \| = 2.$ Since $\lambda k(K)\subseteq [0,1]\lambda,$ $|h(s)|\leq 1$ and $|h(s) -h(t_0)|<\rho$ for every $s\in U,$ and $k|_{_{K\backslash U}}=0$, we apply the above property of $\rho$ to prove that $2=\|h + \lambda k \| \leq  \sqrt{4 - (\theta - \rho)^2}<2$, which is impossible. Therefore, $h(t_0) = \lambda,$ and hence $h\in A(t_0,\lambda)$ and $1=\varphi \Delta(h) =\varphi (-\Delta (f))=-\varphi \Delta (f).$\smallskip

We have seen that $\varphi \Delta(f) = -1,$ for every $f$ in $A(t_0,-\lambda)$ and every $\varphi \in \hbox{supp}(t_0,\lambda)$. Therefore, $ \Delta(A(t_0,-\lambda)) = \varphi^{-1} (\{-1\})\cap \mathcal{B}_{X} = (- \varphi)^{-1} (\{1\}) \cap \mathcal{B}_{X} $ $=- (\varphi)^{-1} (\{1\})\cap \mathcal{B}_{X} $ $= - \Delta(A(t_0,\lambda)),$ for every $\varphi \in \hbox{supp}(t_0,\lambda)$. This shows that $$\hbox{supp}(t_0,-\lambda) = -\hbox{supp}(t_0,\lambda).$$
\end{proof}

The next two results contain a generalized version of \cite[Lemma 2.3 and Proposition 2.4]{JVMorPeRa2017}, the arguments here need an application of Urysohn's lemma.

\begin{lemma}\label{l different points different supports} Suppose $K$ is a compact Hausdorff space. Let $\Delta : S(C(K))\to S(X)$ be a surjective isometry, where $X$ is a complex Banach space. Then the following statements hold:\begin{enumerate}[$(a)$]\item For every $t_0\neq t_1$ in $K$ and every $\lambda,\mu\in \mathbb{T}$ we have $\hbox{supp}(t_0,\lambda) \cap \hbox{supp}(t_1,\mu)=\emptyset;$
\item Given $\mu,\nu\in \mathbb{T}$ with $\mu\neq \nu$, and $t_0$ in $K$, we have $\hbox{supp}(t_0,\nu) \cap \hbox{supp}(t_0,\mu)=\emptyset.$
\end{enumerate}
\end{lemma}

\begin{proof}$(a)$ Arguing by contradiction we assume the existence of $\varphi\in \hbox{supp}(t_0,\lambda) \cap \hbox{supp}(t_1,\mu)$. Let us find, via Urysohn's lemma, two functions $0\leq f_0,f_1\leq 1$ such that $f_0 f_1 =0$ and $f_j (t_j)= 1$ for $j=0,1$. Under these conditions we have  $\lambda f_0\in A(t_0, \lambda)$ and $\mu f_1\in A(t_1, \mu)$.\smallskip

Since $- \mu f_1\in A(t_1, -\mu)$, Lemma \ref{l 3.5 FangWang supports for general compact Hausdorff spaces} implies that  $\varphi \Delta (-\mu f_1) = -1.$ By definition $\varphi \Delta (\lambda f_0)= 1$, and then $$ 2 = \varphi \Delta (\lambda f_0) - \varphi \Delta (- \mu f_1) = |\varphi \Delta (\lambda f_0) - \varphi \Delta (- \mu f_1)  |$$ $$ \leq \| \Delta (\lambda f_0) - \Delta (- \mu f_1) \| =\|\lambda f_0+ \mu f_1\|=1,$$ which is impossible.\smallskip

$(b)$ Arguing as in the previous case, let us take $\varphi\in \hbox{supp}(t_0,\nu) \cap \hbox{supp}(t_0,\mu),$ with $\mu \neq \nu$, and $f_0\in A(t_0,1)$. Since $\mu f_0\in A(t_0,\mu)$ and $\nu f_0\in A(t_0,\nu)$, we get $$2 = \varphi \Delta (\nu f_0) + \varphi \Delta ( \mu f_0) \leq \|\Delta (\nu f_0) + \Delta (\mu f_0) \|\leq 2,$$ and by \cite[Corollary 2.2]{FangWang06} we have $2= \| \nu f_0  +  \mu f_0 \| = |\mu +\nu|,$ which holds if and only if $\mu =\nu$.
\end{proof}

\begin{proposition}\label{p kernel of the support}  Suppose $K$ is a compact Hausdorff space, $X$ is a complex Banach space, and $\lambda\in \mathbb{T}$. Let $\Delta : S(C(K))\to S(X)$ be a surjective isometry. Let $t_0$ be an element in $K$ and let $\varphi$ be an element in $\hbox{supp}(t_0,\lambda)$. Then $\varphi\Delta (f) =0,$ for every $f\in S(C(K))$ with $f(t_0)=0$. Furthermore, $|\varphi\Delta (f) |<1$, for every $f\in S(C(K))$ with $|f(t_0)|<1,$ and every $\varphi\in \hbox{supp}(t_0,\lambda)$.
\end{proposition}

\begin{proof} Let us take $g\in S(C(K))$ such that $g(t)=0$ for every $t$ in an open neighbourhood $U$ of $t_0$. Take, via Urysohn's lemma, a function $f_0\in S(C(K))$ with $0\leq f_0\leq 1$,  $f_0(t_0)=1$ and $f_0|_{K\backslash U}\equiv 0$. The functions $g\pm \lambda f_0\in S(C(K))$ with  $\lambda f_0\in A(t_0,\lambda)$ and $- \lambda f_0\in A(t_0,-\lambda)$. Let us fix $\varphi\in \hbox{supp}(t_0,\lambda)$. Lemma \ref{l 3.5 FangWang supports for general compact Hausdorff spaces} implies that $\varphi \Delta(-\lambda f_0) =-1$, and clearly $\varphi \Delta(\lambda f_0) =1$. Thus $$|\varphi \Delta(g) \pm 1|  = |\varphi \Delta(g) \pm \varphi \Delta(\lambda f_0)| =  |\varphi \Delta(g) - \varphi \Delta(\mp \lambda f_0)|$$ $$\leq \|\varphi\| \ \| \Delta(g) - \Delta(\mp \lambda f_0) \| = \| g  \pm \lambda f_0 \| =1,$$ which assures that  $\varphi \Delta(g)=0.$\smallskip

Since every function $f\in S(C(K))$ with $f(t_0)=0$ can be approximated in norm by functions in $S(C(K))$ vanishing in an open neighbourhood of $t_0$, we deduce from the continuity of $\varphi \Delta$ and the property proved in the previous paragraph that $\varphi \Delta(f) =0$, for every such $f$.\smallskip

For the last statement, let us take $f\in S(C(K))$ with $|f(t_0)|<1,$ and $\varphi\in \hbox{supp}(t_0,\lambda)$. Let us find $1>\varepsilon >0$ such that $|f(t_0)|< 1-\varepsilon$. We consider the non-empty closed set $C_{\varepsilon} :=\{ t\in K : |f(t)|\geq 1-\varepsilon \}$ and the open complement $\mathcal{O}_{\varepsilon} = K\backslash C_{\varepsilon}\ni t_0.$ We can find, via Urysohn's lemma, a function $h\in S(C(K))$ with $0\leq h\leq 1$, $h|_{C_{\varepsilon}} \equiv 1,$ and $h(t_0) =0$. It is easy to check that $f h \in S(C(K))$, $(fh) (t_0)=0$, and $\|f-f h\| \leq 1-\varepsilon <1$.\smallskip

Since $(fh) (t_0)=0$, the first statement of this proposition proves that $\varphi \Delta (fh) =0$, and thus $$| \varphi \Delta (f)  | = | \varphi \Delta (f) -\varphi \Delta (fh) | \leq \| \Delta (f) - \Delta (fh )\| = \| f -fh\| < 1-\varepsilon <1.$$
\end{proof}

Next, we derive a first consequence of the previous proposition.

\begin{corollary}\label{c p kernel of the support}  Suppose $K$ is a compact Hausdorff space, $X$ is a complex Banach space, and $\lambda\in \mathbb{T}$. Let $\Delta : S(C(K))\to S(X)$ be a surjective isometry. If we take $b,c\in S(C(K))$ such that $\Delta(b) = \lambda \Delta(c)$, then $|b(t)|<1,$ for every $t\in K$ satisfying $|c(t)|<1$.
\end{corollary}

\begin{proof} Let us take $t\in K$ satisfying $|c(t)|<1$. By the final statement in Proposition \ref{p kernel of the support} we have $|\varphi\Delta(c)|<1$, for every $\mu\in \mathbb{T}$ and every $\varphi\in \hbox{supp}(t,\mu)$. If $|b(t)|=1$, we can find $\phi\in \hbox{supp}(t,b(t))$ (see Lemma \ref{l existence of support functionals for the image of a face}). Since $b\in A(t,b(t))$, we have $1=\phi \Delta(b) = \phi (\lambda \Delta(c)) = \lambda \phi\Delta(c),$ and thus, $1 = |\lambda| \ |\phi\Delta(c)|<1,$ which leads to a contradiction.
\end{proof}

\section{Geometric properties for Stonean spaces}

For a general compact Hausdorff space $K$, the C$^*$-algebra $C(K)$ rarely contains an abundant collection of projections. For example, $C[0,1]$ only contains trivial projections. If we assume that $K$ is Stonean, then the characteristic function, $\chi_{_A},$ of every non-empty clopen set $A\subset K$ is a continuous function and a projection in $C(K)$, and thus $C(K)$ contains an abundant family of non-trivial projections. Throughout this section we shall work with continuous functions on a Stonean space.\smallskip

Our first result is a reciprocal of Proposition \ref{p kernel of the support} and will be repeatedly applied in our arguments.

%\color{blue}

\begin{proposition}\label{p property of the support}
Suppose $K$ is a Stonean space and $X$ is a complex Banach space. Let $\Delta : S(C(K))\to S(X)$ be a surjective isometry. Let $t_0$ be an element in $K$. If $b$ is an element in $S(C(K))$ satisfying $\varphi\Delta(b)=0,$ for every $\varphi\in \hbox{supp}(t_0,\mu)$ and for every $\mu\in\mathbb{T}$, then $b(t_0)=0$.
\end{proposition}

\begin{proof} Arguing by contradiction, we suppose that $b(t_0)\neq 0$. If $|b(t_0)|= 1$, we can pick $\varphi\in \hbox{supp}(t_0,b(t_0))$ (compare Lemma \ref{l existence of support functionals for the image of a face}). It is clear that $b\in A(t_0,b(t_0)),$ and hence $\varphi \Delta(b) = 1$, which contradicts the hypothesis in the proposition.\smallskip

We deal now with the case $0< |b(t_0)|< 1$. Since $K$ is Stonean, we can always find a clopen subset $W$ satisfying $$t_0\in W\subseteq \left\{s\in K : |b(s)-b(t_0)| < \frac{|b(t_0)|}{2}\right\}.$$ Let us observe that $0<\frac{|b(t_0)|}{2}<|b(s)|,$ for every $s\in W$. Having in mind the last observation, we consider the function $c = b (1-\chi_{_W}) + b |b|^{-1} \chi_{_W}\in C(K)$. Clearly $ \| c\|\leq 1$ and $c(t_0) = \frac{b(t_0)}{|b(t_0)|}\in \mathbb{T}$, therefore $c\in S(C(K))$. It is not hard to check that $$\| c-b\| = \| (b |b|^{-1}-b) \chi_{_W}\| \leq \sup_{s\in W} | b(s) (|b(s)|^{-1}-1) | =\sup_{s\in W} |1-|b(s)|| $$ $$\leq  1 - \inf_{s\in W} |b(s)| \leq 1-\frac{|b(t_0)|}{2}.$$ The element $c$ lies in $A\left(t_0, \frac{b(t_0)}{|b(t_0)|}\right)$, and so we can conclude, by taking $\mu\in\mathbb{T}$, $\varphi\in \hbox{supp}(t_0,\mu)$ and applying the hypothesis, that $$1=\varphi \Delta(c)-\varphi\Delta(b)\leq \|\Delta(c)- \Delta(b) \|=\|c-b\|\leq 1-\frac{|b(t_0)|}{2},$$ leading to  $\frac{|b(t_0)|}{2}\leq 0$, which is impossible.	
\end{proof}

Our next results are devoted to determine the behaviour of a surjective isometry $\Delta : S(C(K))\to S(X)$ on elements which are finite linear combinations of mutually orthogonal projections. We begin with a single characteristic function of a clopen set.

\begin{proposition}\label{p characteristic of clopen sets 1}  Suppose $K$ is a Stonean space, $A$ is a non-empty clopen subset of $K$, $X$ is a complex Banach space, and $\lambda, \gamma\in \mathbb{T}$. Let $\Delta : S(C(K))\to S(X)$ be a surjective isometry. If we take $b\in S(C(K))$ such that $\Delta(b) = \lambda \Delta(\gamma \chi_{_A})$, then $b = b \chi_{_A}$ and $|b(t)|=1,$ for every $t\in A$.
\end{proposition}

\begin{proof} We shall first prove that $b = b \chi_{_A}$. Let us fix $t_0\in K\backslash A$. If we pick an arbitrary $\mu\in\mathbb{T}$ and $\varphi\in \hbox{supp}(t_0,\mu)$, combining the hypothesis with Proposition \ref{p kernel of the support} we get $\varphi\Delta(b)=\lambda\varphi\Delta(\gamma\chi_{_{A}})=0$ which implies, via Proposition \ref{p property of the support}, that $b(t_0)=0$. The arbitrariness of $t_0$ guarantees that $b = b \chi_{_A}$.\smallskip

Take now $t_0\in A$. If $|b(t_0)|<1$, the second statement in Proposition \ref{p kernel of the support} assures that $|\varphi \Delta(b)|<1,$ for every $\varphi\in \hbox{supp} \left(t_0, \gamma  \right)$. However, in this case, $1> |\varphi \Delta(b)| = |\varphi (\lambda \Delta(\gamma  \chi_{_A}))| = |\lambda| |\varphi \Delta(\gamma  \chi_{_A})| = 1,$ which leads to a contradiction. Therefore, $|b(t_0)|=1$, for every $t_0\in A$.
\end{proof}

The next lemma is an elementary technical observation with a curious geometric interpretation.

\begin{lemma}\label{l technical distances in the sphere of C} Let $\delta$ be a real number with $0< \delta <2$. Then the set $$\{\zeta \in \mathbb{T} : |\zeta -1|^2 \geq \delta^2, \ |\zeta +1|^2 \geq 4-\delta^2 \}$$  coincides with $\{\lambda, \overline{\lambda}\}$ for a unique $\lambda \in \mathbb{T}$ with $|\zeta -1|^2 =\delta^2,$ and $|\zeta +1|^2 = 4-\delta^2$. Moreover, for each $\gamma\in \mathbb{T}$ we have $$\{\zeta \in \mathbb{T} : |\zeta-\gamma|^2 \geq |\lambda -1|^2 , \ |\zeta +\gamma|^2 \geq |\lambda +1|^2 \} = \{\lambda \gamma, \overline{\lambda} \gamma\}.$$
\end{lemma}

\begin{proof} Let us take $0<\delta <2$. It is standard to prove that the set
$Z=\{\zeta \in \mathbb{T} : |\zeta -1|^2 \geq \delta^2, \ |\zeta +1|^2 \geq 4-\delta^2 \}$ is composed of just one complex number and its conjugate, both of them depending only on $\delta$. Actually, if we solve the corresponding system of inequalities associated to the conditions required to be in $Z$, we find that the only two analytic solutions are $\lambda=\frac{1}{2}(2-\delta^2+ i \delta\sqrt{4-\delta^2})$ and $\overline{\lambda}=\frac{1}{2}(2-\delta^2- i \delta\sqrt{4-\delta^2})$. It is worth to observe that $Z$ is precisely the set of those elements in the complex unit sphere which are outside the open disc of center $(1,0)$ and with radius $\delta$ and outside the open disc of center $(-1,0)$ and radius $\sqrt{4-\delta^2}$. Figure \ref*{circumferences} below illustrates this geometric interpretation.\smallskip

\begin{figure}[h]
	\centering
	\includegraphics[width=0.4\textwidth]{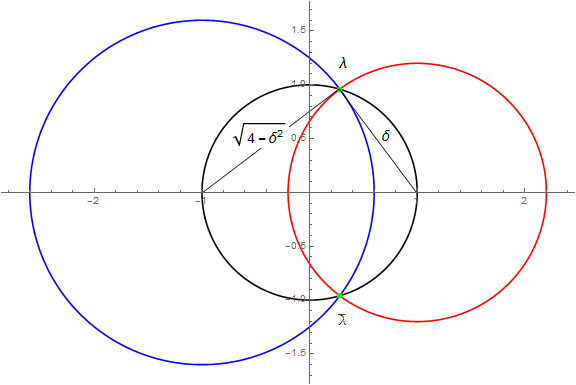}
	\caption{Particular case of Lemma \ref{l technical distances in the sphere of C} with $\delta=1.2$}
	\label{circumferences}
\end{figure}

According to the above observations, for each $\gamma\in\mathbb{T}$, the set $\{\zeta \in \mathbb{T} : |\zeta-\gamma|^2\geq |\lambda -1|^2 , \ |\zeta +\gamma|^2 \geq |\lambda +1|^2 \}$ can be identified with an appropriate turn of $Z.$ The parameter $\delta$ is exactly the distance from $\lambda$ to $1$ and $|\lambda + 1|^2= 4-\delta^2$. In this new setting, we work with the complex sphere and the circumferences centered at $\gamma$ and $-\gamma$ with radii $\delta$ and $\sqrt{4-\delta^2}$, respectively. %by the parallelogram law.
Thus, the only two elements in this turned set are $\lambda\gamma$ and $\overline{\lambda}\gamma$.
\end{proof}

We can now complete the information in Proposition \ref{p characteristic of clopen sets 1}. Henceforth, for each element $a$ in a complex Banach algebra $A$, the symbol $\sigma(a)$ will stand for the spectrum of $a$.

\begin{proposition}\label{p characteristic of clopen sets 2}  Suppose $K$ is a Stonean space, $A$ is a non-empty clopen subset of $K$, $\lambda,\gamma\in \mathbb{T}$, and $X$ is a complex Banach space. Let $\Delta : S(C(K))\to S(X)$ be a surjective isometry. If we take $b\in S(C(K))$ such that $\Delta(b) = \lambda \Delta(\gamma \chi_{_A})$, then $b = b \chi_{_A}$ and $\sigma(b)\subseteq \{\lambda \gamma ,\overline{\lambda} \gamma ,0\}$. Consequently, there exist two disjoint clopen sets $A_1$ and $A_2$ {\rm(}one of which could be empty{\rm)} such that $A = A_1 \cup A_2$ and $ b = \lambda \gamma  \chi_{_{A_1}} + \overline{\lambda} \gamma  \chi_{_{A_2}}$. Consequently, $\Delta(- \gamma  \chi_{_A}) = - \Delta( \gamma \chi_{_A})$.
\end{proposition}

\begin{proof} Proposition \ref{p characteristic of clopen sets 1} implies that $b = b \chi_{_A}$ and $|b(t)|=1,$ for every $t\in A$.\smallskip

We assume first that $\lambda\neq \pm 1$ (i.e. $|\lambda -1|,|\lambda+1|\in (0,2)$ and $|\lambda -1|^2+|\lambda+1|^2 =4$).\smallskip

We fix an arbitrary $t_0\in A$. Let us observe the following property: For each $\varphi \in \hbox{supp}(t_0,\gamma),$ we have $\varphi\Delta(b) = \varphi (\lambda \Delta(\gamma \chi_{_A})) = \lambda  \varphi \Delta(\gamma \chi_{_A}) = \lambda$. Therefore, by Lemma \ref{l 3.5 FangWang supports for general compact Hausdorff spaces}, for each $g\in A(t_0, \gamma)$ and each $k\in A(t_0,-\gamma)$, we have $$|\lambda - 1| = | \varphi\Delta(b) - \varphi\Delta(g)| \leq \| \Delta(b) -  \Delta(g)\| = \|b- g\|,$$ and $$|\lambda + 1| = | \varphi\Delta(b) - \varphi\Delta(k)| \leq \| \Delta(b) -  \Delta(k)\| = \|b- k\|.$$ Since $A(t_0,\gamma) = -A(t_0,-\gamma),$ it follows that \begin{equation}\label{eq distance bigger than or equal to sqrt2}  |\lambda - 1| \leq  \|b- g\|, \hbox{ and } |\lambda + 1| \leq  \|b+ g\|, \hbox{ for all } g\in A(t_0,\gamma).
\end{equation}

For each $0<\varepsilon <1$, let us find a clopen set $W$ satisfying $$t_0\in W\subset \{s\in K : |b(s)-b(t_0)| <\varepsilon\}.$$ We consider the functions $g^{\pm}_{\varepsilon} = \pm b(1-\chi_{_W}) + \gamma \chi_{_W}\in S(C(K)),$ which clearly lie in $A(t_0,\gamma)$. By \eqref{eq distance bigger than or equal to sqrt2} we have $$|\lambda - 1| \leq \|b- g_{\varepsilon}^+\| = \sup_{s\in W} |b(s)- \gamma| \leq |b(t_0)- \gamma| + \varepsilon,$$
and
$$|\lambda + 1| \leq \|b+ g_{\varepsilon}^-\| = \sup_{s\in W}|b(s)+ \gamma| \leq |b(t_0)+ \gamma| + \varepsilon,$$
which implies that $|b(t_0)\pm \gamma| \geq |\lambda \pm 1|-\varepsilon.$ The arbitrariness of $0<\varepsilon <1$ gives $|b(t_0)\pm \gamma| \geq |\lambda \pm 1|.$ Since $|b(t_0)|=1$, we conclude that $b(t_0)\in \{\lambda \gamma,\overline{\lambda} \gamma\},$ for every $t_0\in A$ (cf. Lemma \ref{l technical distances in the sphere of C}). We have therefore shown that $\sigma(b)= b(K) \subseteq \{\lambda \gamma,\overline{\lambda} \gamma,0\}.$ The rest is clear.\smallskip

We deal now with $\lambda = \pm 1$. The statement is clear for $\lambda= 1$ with $b= \gamma \chi_{_A}$. Finally, let us assume that $\lambda =-1$. We fix an arbitrary $t_0\in A$. By repeating the previous arguments, or by Lemma \ref{l 3.5 FangWang supports for general compact Hausdorff spaces}, we deduce that, for each $\varphi \in \hbox{supp}(t_0,\gamma),$ we have $\varphi\Delta(b) = -1$, and thus $$  2=|-1 - 1| = | \varphi\Delta(b) - \varphi\Delta(g)| \leq \| \Delta(b) -  \Delta(g)\| = \|b- g\| \leq 2,$$ for every $g\in A(t_0,\gamma)$. As before, given $0<\varepsilon <1$, we consider a clopen set $W$ such that $t_0\in W\subset \{s\in K : |b(s)-b(t_0)| <\varepsilon\},$ and the function $g^{+}_{\varepsilon} =  b(1-\chi_{_W}) + \gamma \chi_{_W}\in A(t_0,\gamma)$. Since $$2 = \|b- g^{+}_{\varepsilon}\| = \sup_{s\in W} |b(s)- \gamma| \leq |b(t_0)- \gamma| + \varepsilon,$$ we deduce from the arbitrariness of $\varepsilon>0$ that $2\leq |b(t_0)- \gamma| \leq 2$, and thus $b(t_0)=- \gamma$. We have shown that $b(t_0)=- \gamma$ for every $t_0\in A$.
\end{proof}

We recall that a set $\{x_1,\ldots,x_k\}$ in a complex Banach space $X$ is called \emph{completely $M$-orthogonal} if $$\Big\|\sum_{j=1}^k \alpha_j x_j \Big\| = \max\{ \|\alpha_j x_j\| : 1\leq j\leq k\},$$ for every $\alpha_1,\ldots,\alpha_k$ in $\mathbb{C}$. If $\{x_1,\ldots,x_k\} \subset S(X)$, then it is completely $M$-orthogonal if and only if the equality $$\displaystyle \left\|\sum_{j=1}^k \lambda_j x_j \right\| = 1$$ holds for every $\lambda_1,\ldots,\lambda_k$ in $\mathbb{T}$ and $\lambda_{j_0}=1$ for some $j_0\in \{1,\ldots,k\}$  (see \cite[Lemma 3.4]{JVMorPeRa2017} and \cite[Lemma 2.3]{Pe2017}).\smallskip

We can now complete the information given in Proposition \ref{p characteristic of clopen sets 2}.

\begin{proposition}\label{p characteristic of clopen sets 3 for two disjoint}Suppose $K$ is a Stonean space. Let $A$ and $B$ be two non-empty disjoint clopen subsets of $K$.  Let $\Delta : S(C(K))\to S(X)$ be a surjective isometry, where $X$ is a complex Banach space. Then the following statements hold:
\begin{enumerate}[$(a)$]\item For every $\gamma,\mu\in \mathbb{T}$, the set $\{\Delta (\gamma \chi_{_A}),\Delta (\mu \chi_{_B})\}$ is completely $M$-orthogonal;
\item $\Delta (\sigma_1 \gamma \chi_{_A}+\sigma_2 \mu \chi_{_B}) = \sigma_1 \Delta (\gamma \chi_{_A})+\sigma_2 \Delta (\mu \chi_{_B})$, for every $\sigma_1,\sigma_2\in \{\pm 1\}$ and every $\gamma,\mu\in \mathbb{T}$;
\item For each $\lambda\in \mathbb{T}$, there exist two disjoint clopen sets $A_1$ and $A_2$ {\rm(}one of which could be empty{\rm)} such that $A = A_1 \cup A_2$, $$ \lambda \Delta( \chi_{_{A_1}} )+{\lambda}   \Delta( \chi_{_{A_2}}) = \Delta( \lambda \chi_{_{A_1}} )+  \Delta(\overline{\lambda}  \chi_{_{A_2}}) = \Delta( \lambda \chi_{_{A_1}} + \overline{\lambda}  \chi_{_{A_2}}) = \lambda \Delta(\chi_{_A}),$$ $\Delta( \lambda \chi_{_{A_1}}) =  \lambda \Delta( \chi_{_{A_1}})$, $\Delta( \lambda \chi_{_{A_2}} ) =\overline{\lambda} \Delta( \chi_{_{A_2}}),$
    $$ \overline{\lambda} \Delta( \chi_{_{A_1}} )+ \overline{\lambda}   \Delta( \chi_{_{A_2}}) = \Delta( \overline{\lambda} \chi_{_{A_1}} )+  \Delta({\lambda}  \chi_{_{A_2}}) = \Delta( \overline{\lambda} \chi_{_{A_1}} + {\lambda}  \chi_{_{A_2}}) = \overline{\lambda} \Delta(\chi_{_A}),$$
    $\Delta( \overline{\lambda} \chi_{_{A_1}}) =  \overline{\lambda} \Delta( \chi_{_{A_1}})$,
    and $\Delta( \overline{\lambda} \chi_{_{A_2}} ) ={\lambda} \Delta( \chi_{_{A_2}}).$
\end{enumerate}
\end{proposition}

\begin{proof}$(a)$ Let us take $\lambda,\mu, \gamma \in \mathbb{T}$. By Proposition \ref{p characteristic of clopen sets 2} there exist two disjoint clopen sets $A_1$ and $A_2$ such that $A = A_1 \cup A_2$ and $ \Delta( \lambda \gamma \chi_{_{A_1}} + \overline{\lambda} \gamma \chi_{_{A_2}}) = \lambda \Delta(\gamma \chi_{_A}).$ Therefore, by the hypothesis, we have $$\left\| \Delta(\mu \chi_{_B}) \pm \lambda \Delta(\gamma \chi_{_A}) \right\|= \left\|\Delta(\mu \chi_{_B}) - \Delta(\mp \lambda \gamma \chi_{_{A_1}} \mp \overline{\lambda} \gamma \chi_{_{A_2}}) \right\| $$ $$ = \left\|\mu\chi_{_B} \pm  \lambda \gamma \chi_{_{A_1}} \pm \overline{\lambda} \gamma \chi_{_{A_2}} \right\| =1, $$ which proves the statement.\smallskip

$(b)$ Let us fix $\sigma_1,\sigma_2\in \{\pm 1\}.$ Since, by $(a)$, $\{\Delta(\gamma \chi_{_{A}}),\Delta(\mu \chi_{_{B}})\}$ is completely $M$-orthogonal, it follows that $\sigma_1 \Delta(\gamma \chi_{_{A}} )+\sigma_2 \Delta(\mu \chi_{_{B}})\in S(X)$, and thus there exists $b\in S(C(K))$ satisfying $\Delta(b) = \sigma_1 \Delta(\gamma \chi_{_{A}} )+\sigma_2 \Delta(\mu \chi_{_{B}}).$ If we take $t_0\in K\backslash (A\cup B)$, an arbitrary element $\alpha$ of $\mathbb{T}$ and $\varphi\in\hbox{supp}(t_0,\alpha)$, then we have, via Proposition \ref{p kernel of the support}, that $\varphi\Delta(b) = \sigma_1 \varphi\Delta(\gamma \chi_{_{A}} )+\sigma_2 \varphi\Delta(\mu \chi_{_{B}})=0$ and Proposition \ref{p property of the support} concludes that $b= b\chi_{A\cup B}$ because of the arbitrariness of $t_0$.  By repeating the arguments in the proof of Proposition \ref{p characteristic of clopen sets 1} we get $|b(t) |=1,$ for all $t\in A\cup B$.\smallskip

Pick $t_0\in A$ and $\varphi \in \hbox{supp}(t_0,\sigma_1 \gamma)$. By Proposition \ref{p kernel of the support} we have $\varphi \Delta (b) = \varphi (\sigma_1 \Delta(\gamma  \chi_{_{A}})  + \sigma_2 \Delta(\mu \chi_{_{B}})) = 1,$ and hence $\Delta(b) \in \varphi^{-1} (\{1\})\cap \mathcal{B}_{X} = \Delta(A(t_0,\sigma_1 \gamma))$ (cf. Lemma \ref{l existence of support functionals for the image of a face}). This shows that $b(t_0) = \sigma_1 \gamma$ for all $t_0\in A$. Similarly, $b(t_0) = \sigma_2\mu $ for all $t_0\in B$. We have therefore shown that $b= \sigma_1 \gamma \chi_{_A}+\sigma_2 \mu \chi_{_B}$ and $$\Delta (\sigma_1 \gamma \chi_{_A}+\sigma_2 \mu \chi_{_B}) = \sigma_1 \Delta (\gamma \chi_{_A})+\sigma_2 \Delta (\mu \chi_{_B}).$$

$(c)$ We may assume that $\lambda \neq \pm 1$. Proposition \ref{p characteristic of clopen sets 2} proves the existence of two disjoint clopen sets $A_1$ and $A_2$ such that $A = A_1 \cup A_2$ and $$\Delta( \lambda \chi_{_{A_1}})  + \Delta(\overline{\lambda}  \chi_{_{A_2}}) = \Delta( \lambda \chi_{_{A_1}} + \overline{\lambda}  \chi_{_{A_2}}) = \lambda \Delta(\chi_{_A}) = \lambda \Delta(\chi_{_{A_1}}) + \lambda \Delta(\chi_{_{A_2}}),$$ where in the first and last equalities we have applied $(b)$. Therefore, $$\Delta( \lambda \chi_{_{A_1}})- \lambda \Delta(\chi_{_{A_1}})   =  \lambda \Delta(\chi_{_{A_2}}) - \Delta(\overline{\lambda}  \chi_{_{A_2}}). $$ We deduce from this identity and Proposition \ref{p kernel of the support} that \begin{equation}\label{eq 17012018} \varphi (\Delta( \lambda \chi_{_{A_1}})- \lambda \Delta(\chi_{_{A_1}}))   =  \varphi (\lambda \Delta(\chi_{_{A_2}}) - \Delta(\overline{\lambda}  \chi_{_{A_2}}))= 0,
 \end{equation} for every $t_0\in A_1,$ $\mu\in \mathbb{T}$ and $\varphi \in \hbox{supp} (t_0,\mu)$. However, a new application of Proposition \ref{p characteristic of clopen sets 2} assures the existence of disjoint clopen sets $A_{11}$ and $A_{12}$ such that $A_1 = A_{11} \cup A_{12},$ and $ \Delta( \lambda \chi_{_{A_{11}}} + \overline{\lambda}  \chi_{_{A_{12}}}) = \lambda \Delta(\chi_{_{A_{1}}})$, and by $(b)$, we get $$ \Delta( \lambda \chi_{_{A_{11}}}) + \Delta(\overline{\lambda}  \chi_{_{A_{12}}}) = \lambda \Delta(\chi_{_{A_{11}}}) + \lambda \Delta(\chi_{_{A_{12}}}).$$ If we can find $t_0\in A_{12}$, then by \eqref{eq 17012018}, Proposition \ref{p kernel of the support}, and $(b)$ we have $\varphi \Delta(\lambda \chi_{_{A_{12}}}) = \varphi \Delta ( \overline{\lambda} \chi_{_{A_{12}}})=1$, for every $\varphi\in \hbox{supp}(t_0, \overline{\lambda})$. Consequently, $$2 = \varphi \Delta(\lambda \chi_{_{A_{12}}}) + \varphi \Delta(g) \leq \|\Delta(\lambda \chi_{_{A_{12}}}) + \Delta(g) \| \leq 2,$$ for all $g\in A(t_0,\overline{\lambda})$. Corollary 2.2 in \cite{FangWang06} implies that $\|\lambda \chi_{_{A_{12}}} + g \| = 2,$ for all $g\in A(t_0,\overline{\lambda})$. In particular, for every clopen $W\subset A_{12}$ with $t_0\in W$ (taking $g= \overline{\lambda} \chi_{_W}$) we deduce the existence of $s_{_W}\in W$ such that $|\lambda + \overline{\lambda} | =2,$ which is impossible. Therefore $A_{12} =\emptyset,$ and thus $\Delta(\lambda \chi_{_{A_1}})= \lambda \Delta(\chi_{_{A_1}})$.\smallskip

Similar arguments lead to $\Delta(\lambda \chi_{_{A_2}})= \overline{\lambda} \Delta(\chi_{_{A_2}}).$\smallskip

We shall finally prove the last identities. By the above arguments there exist disjoint clopen sets $A_3$ and $A_4$ {\rm(}one of which could be empty{\rm)} such that $A = A_3 \cup A_4$, $$ \overline{\lambda} \Delta( \chi_{_{A_3}} )+ \overline{\lambda}   \Delta( \chi_{_{A_4}}) = \Delta( \overline{\lambda} \chi_{_{A_3}} )+  \Delta({\lambda}  \chi_{_{A_4}}) = \Delta( \overline{\lambda} \chi_{_{A_3}} + {\lambda}  \chi_{_{A_4}}) = \overline{\lambda} \Delta(\chi_{_A}),$$
    $\Delta( \overline{\lambda} \chi_{_{A_3}}) =  \overline{\lambda} \Delta( \chi_{_{A_3}})$,
    and $\Delta( \overline{\lambda} \chi_{_{A_4}} ) ={\lambda} \Delta( \chi_{_{A_4}}).$ We shall finish by proving that $A_1 = A_3$ and $A_2= A_4$.
If there exists $t_0\in A_1\cap A_4$, we pick $\varphi\in \hbox{supp}(t_0, \lambda)$ and, by Proposition \ref{p kernel of the support}, we compute $$\varphi \Delta(\chi_{_{A}}) = \varphi \left(\lambda \Delta( \overline{\lambda} \chi_{_{A_3}} + {\lambda}  \chi_{_{A_4}}) \right) = \varphi \left(\lambda  \Delta( \overline{\lambda} \chi_{_{A_3}} )+ \lambda \Delta({\lambda}  \chi_{_{A_4}}) \right) =\lambda,$$ and $$\varphi \Delta(\chi_{_{A}}) = \varphi \left(\overline{\lambda} \Delta( {\lambda} \chi_{_{A_1}} + \overline{\lambda}  \chi_{_{A_2}}) \right) = \varphi \left(\overline{\lambda}  \Delta( {\lambda} \chi_{_{A_1}} )+ \overline{\lambda} \Delta(\overline{\lambda}  \chi_{_{A_2}}) \right) =\overline{\lambda},$$ which is impossible because $\lambda\neq \pm 1$. This shows that $A_1=A_3$ and $A_2=A_4$.
\end{proof}

An appropriate generalization of \cite[Proposition 3.3]{JVMorPeRa2017} is established next.

\begin{proposition}\label{p characteristic of clopen sets 3}  Suppose $K$ is a Stonean space, $A$ is a non-empty clopen subset of $K$, $\lambda\in \mathbb{T}\backslash \mathbb{R}$, and $X$ is a complex Banach space. Let $\Delta : S(C(K))\to S(X)$ be a surjective isometry. We additionally assume that $\Delta(\lambda\chi_{_A}) = \lambda \Delta(\chi_{_A})$ {\rm(}respectively, $\Delta(\lambda\chi_{_A}) = \overline{\lambda} \Delta(\chi_{_A})${\rm)}. Then $\Delta(\mu\chi_{_A}) = \mu \Delta(\chi_{_A})$ {\rm(}respectively, $\Delta(\mu\chi_{_A}) = \overline{\mu} \Delta(\chi_{_A})${\rm)}, for every $\mu \in \mathbb{T}$. Furthermore, if $B$ is another non-empty clopen set in $K$ contained in $A$, then $\Delta(\mu\chi_{_B}) = \mu \Delta(\chi_{_B})$ {\rm(}respectively, $\Delta(\mu\chi_{_B}) = \overline{\mu} \Delta(\chi_{_B})${\rm)}, for every $\mu \in \mathbb{T}$.
\end{proposition}

\begin{proof} We shall only prove the case in which $\Delta(\lambda\chi_{_A}) = \lambda \Delta(\chi_{_A})$, the other statement is very similar. Let us take $\mu\in\mathbb{T}$. If $\mu=\pm 1$, then it is clear that the statement holds by Proposition \ref{p characteristic of clopen sets 2}. We can therefore assume that $\mu\in\mathbb{T}\backslash \RR$. Proposition \ref{p characteristic of clopen sets 3 for two disjoint}$(c)$ proves the existence of two disjoint clopen sets $A_1$ and $A_2$ (one of which could be empty) such that $A=A_1\cup A_2$,
$$ \mu \Delta( \chi_{_{A_1}} )+\mu   \Delta( \chi_{_{A_2}}) = \Delta( \mu \chi_{_{A_1}} )+  \Delta(\overline{\mu}  \chi_{_{A_2}}) = \Delta( \mu \chi_{_{A_1}} + \overline{\mu}  \chi_{_{A_2}}) = \mu \Delta(\chi_{_A}),$$ $\Delta( \mu \chi_{_{A_1}}) =  \mu \Delta( \chi_{_{A_1}})$, and $\Delta( \mu \chi_{_{A_2}} ) =\overline{\mu} \Delta( \chi_{_{A_2}}).$\smallskip

We claim that $A_2=\emptyset$. Otherwise, by Proposition \ref{p characteristic of clopen sets 3 for two disjoint} we have
$$|\lambda+\mu|= \|\lambda \Delta(\chi_{_{A}}) + \mu \Delta(\chi_{_{A}}) \|=\| \Delta( \lambda \chi_{_{A_1}} )+ \Delta(\lambda   \chi_{_{A_2}}) + \Delta(\mu \chi_{_{A_1}} )+ \Delta( \overline{\mu}  \chi_{_{A_2}}) \| $$ $$
=\max\{ \| \Delta( \lambda \chi_{_{A_1}} )+  \Delta(\mu \chi_{_{A_1}} ) \|, \|  \Delta(\lambda   \chi_{_{A_2}}) +  \Delta( \overline{\mu}  \chi_{_{A_2}}) \|\} =\hbox{(by Proposition \ref{p characteristic of clopen sets 2})}$$
$$=\max\{ \| \lambda \chi_{_{A_1}} + \mu \chi_{_{A_1}}  \|, \| \lambda   \chi_{_{A_2}} + \overline{\mu}  \chi_{_{A_2}} \|\}= \max\{ |\lambda+\mu|, |\lambda+\overline{\mu}|\},$$ and hence $|\lambda+\overline{\mu}|\leq |\lambda+{\mu}|$. By replacing $\mu$ with $-\mu$ in the above arguments we get $|\lambda-\overline{\mu}|\leq |\lambda-{\mu}|$. Combining the last two inequalities we have $\Re\hbox{e}(\lambda\overline{\mu})= \Re\hbox{e}(\lambda\mu),$ or equivalently, $\lambda\overline{\mu}+\overline{\lambda}\mu = \lambda\mu + \overline{\lambda}\overline{\mu}$, which holds if and only if $\mu(\overline{\lambda}-\lambda)=\overline{\mu}(\overline{\lambda}-\lambda)$ and $\lambda(\overline{\mu}-\mu)=\overline{\lambda}(\overline{\mu}-\mu)$. The last equalities hold if and only if $\lambda,\mu\in\RR$, which is impossible.

For the second statement, let us take a non-empty clopen set $B$ in $K$ contained in $A$. We assume $\Delta(\lambda\chi_{_A}) = \lambda \Delta(\chi_{_A})$ {\rm(}respectively, $\Delta(\lambda\chi_{_A}) = \overline{\lambda} \Delta(\chi_{_A})${\rm)}. The desired equality is clear if $\mu=\pm 1,$ we thus assume that $\mu\in\mathbb{T}\backslash\RR$. Proposition \ref{p characteristic of clopen sets 3 for two disjoint}$(c)$ guarantees the existence of two disjoint clopen sets $B_1,B_2$ in $K$ such that $B=B_1\cup B_2$ and $\mu\Delta(\chi_{_B}) = \Delta(\mu\chi_{_{B_1}})+ \Delta(\overline{\mu}\chi_{_{B_2}}) $. Observe that $A=B_1\cup B_2 \cup (A\backslash B)$ and that $A\backslash B=A\cap (K\backslash B)$ is a clopen set in $K$. We therefore have $$\mu\Delta(\chi_{_{A}})=\mu\Delta(\chi_{_{B_1}})+\mu\Delta(\chi_{_{B_2}})+\mu\Delta(\chi_{_{A\backslash B}})=\Delta(\mu\chi_{_{B_1}})+\Delta(\overline{\mu}\chi_{_{B_2}})+\mu\Delta(\chi_{_{A\backslash B}})$$ and, by applying the first conclusion in this proposition and Proposition \ref{p characteristic of clopen sets 3 for two disjoint} we deduce that $$\mu\Delta(\chi_{_{A}})=\Delta(\mu\chi_{_{A}})=\Delta(\mu\chi_{_{B_1}})+\Delta(\mu\chi_{_{B_2}})+\Delta(\mu\chi_{_{A\backslash B}})$$ (respectively, $\mu\Delta(\chi_{_{A}})=\Delta(\overline{\mu}\chi_{_{A}})=\Delta(\overline{\mu}\chi_{_{B_1}})+\Delta(\overline{\mu}\chi_{_{B_2}})+\Delta(\overline{\mu}\chi_{_{A\backslash B}})$). Then the identity
$$
\Delta (\overline{\mu}\chi_{_{B_2}})+\mu\Delta (\chi_{_{A\backslash B}})=\Delta (\mu\chi_{_{B_2}})+\Delta (\mu\chi_{_{A\backslash B}})
$$
(respectively,
$$
\Delta (\mu\chi_{_{B_1}})+\mu\Delta (\chi_{_{A\backslash B}})=\Delta (\overline{\mu}\chi_{_{B_1}})+\Delta (\overline{\mu}\chi_{_{A\backslash B}}))
$$ holds. Therefore, $\varphi(\Delta (\overline{\mu}\chi_{_{B_2}})-\Delta (\mu\chi_{_{B_2}}))=\varphi(\Delta (\mu\chi_{_{A\backslash B}})-\mu\Delta (\chi_{_{A\backslash B}}))=0,$ for every $t_0\in B_2$, $\gamma\in \mathbb{T}$ and $\varphi\in\hbox{supp}(t_0, \gamma)$. If we can find $t_0\in B_2$, then for $\gamma=\mu$ we have $\varphi(\Delta(\overline{\mu}\chi_{_{B_2}}))=\varphi(\Delta(\mu\chi_{_{B_2}}))=1,$ and hence $\Delta(\overline{\mu}\chi_{_{B_2}})\in \varphi^{-1} (\{1\})\cap \mathcal{B}_{X} = \Delta(A(t_0,\mu))$ (cf. Lemma \ref{l existence of support functionals for the image of a face}). Thus $\overline{\mu}\chi_{_{B_2}}\in A(t_0,\mu)$, which is impossible. We have shown that $B_2=\emptyset,$ and hence $B=B_1$ and $\Delta(\mu\chi_{_B}) = \mu \Delta(\chi_{_B})$.\smallskip

In the case $\Delta(\lambda\chi_{_A}) = \overline{\lambda} \Delta(\chi_{_A})$, similar arguments prove that $\Delta(\mu\chi_{_B}) = \overline{\mu}\Delta(\chi_{_B})$.
\end{proof}

A first corollary of the above proposition plays a fundamental role in our argument.

\begin{corollary}\label{c characteristic of clopen sets 2 a}  Suppose $K$ is a Stonean space and $X$ is a complex Banach space. Let $\Delta : S(C(K))\to S(X)$ be a surjective isometry. Then there exists a clopen subset $K_1\subseteq K$ such that $\Delta(\lambda \chi_{_{K_1}}) =\lambda \Delta(\chi_{_{K_1}})$ and $\Delta(\lambda \chi_{_{K\backslash K_1}}) =\overline{\lambda} \Delta(\chi_{_{K\backslash K_1}})$, for every $\lambda\in \mathbb{T}$. Consequently, if $B_1$ is a clopen subset of $K$ contained in $K_1$ and $B_2$ is a clopen subset of $K$ contained in $K_2=K\backslash K_1$, then $\Delta(\mu\chi_{_{B_1}})=\mu\Delta(\chi_{_{B_1}})$ and $\Delta(\mu\chi_{_{B_2}})=\overline{\mu}\Delta(\chi_{_{B_2}}),$ for every $\mu\in\mathbb{T}$.
\end{corollary}

\begin{proof} The proof follows straightforwardly from Propositions \ref{p characteristic of clopen sets 3 for two disjoint} and \ref{p characteristic of clopen sets 3}.
\end{proof}

From now on, given a surjective isometry $\Delta : S(C(K))\to S(X)$ where $K$ is a Stonean space and $X$ is a complex Banach space, the symbols $K_1$ and $K_2$ will denote the clopen subsets given by Corollary \ref{c characteristic of clopen sets 2 a}. Under these hypothesis we define a new product $\odot: \mathbb{C}\times C(K)\to C(K)$ given by \begin{equation}\label{eq def product odot} (\alpha \odot a) (t) := \alpha\ a(t), \hbox{ if } t\in K_1,\hbox{ and } (\alpha \odot a) (t) := \overline{\alpha}\  a(t), \hbox{ otherwise}.
\end{equation} We observe that $\alpha \odot a= \alpha \ a$ whenever $\alpha\in \mathbb{R}$.\smallskip

Our next results are devoted to determine the behaviour of a surjective isometry $\Delta : S(C(K))\to S(X)$ on algebraic elements.

\begin{proposition}\label{p image of addition 1}
Suppose $K$ is a Stonean space, $\gamma_1,\ldots,\gamma_n\in \mathbb{T}$, $B_1,\ldots,B_n$ are non-empty disjoint clopen subsets of $K,$ and $X$ is a complex Banach space. Let $\Delta : S(C(K))\to S(X)$ be a surjective isometry and let $\displaystyle v=\sum_{k=1}^{m}\lambda_k\chi_{_{A_{k}}}$ be an algebraic partial isometry in $C(K)$, where $\lambda_1, \ldots, \lambda_m\in \mathbb{T}$, $A_1,\ldots,A_m$ are non-empty disjoint clopen sets in $K$ such that $A_k\cap B_j=\emptyset$, for every $k\in\{1,\ldots,m\}$ and every $ j\in\{1,\ldots,n\}$. Then the set $\{ \Delta(v), \Delta(\gamma_1\chi_{_{B_{1}}}), \ldots, \Delta(\gamma_n\chi_{_{B_{n}}}) \}$ is completely $M$-orthogonal, and the equality $$\Delta(v) + \sum_{j=1}^{n}\Delta(\gamma_j\chi_{_{B_{j}}}) = \Delta\left(v + \sum_{j=1}^{n}\gamma_j\chi_{_{B_{j}}}\right)$$ holds.
\end{proposition}

\begin{proof}
We shall prove the statement arguing by induction on $n$. In the case $n=1,$ let us take $\mu_1\in\mathbb{T}$. Since $B_1$ is a non-empty clopen set, by Proposition \ref{p characteristic of clopen sets 2} there exist two disjoint clopen sets $B_{11}$ and $B_{12}$ such that $B_1=B_{11}\cup B_{12}$ and $\mu_1\Delta(\gamma_1\chi_{_{B_{1}}})=\Delta( \mu_1\gamma_1\chi_{_{B_{11}}} + \overline{\mu_1}\gamma_1\chi_{_{B_{12}}} )$. Since $\chi_{_{B_1}}$ is orthogonal to $v$, it follows from Proposition \ref{p characteristic of clopen sets 3 for two disjoint} and the hypotheses that
$$\| \Delta(v) + \mu_1\Delta(\gamma_1\chi_{_{B_{1}}}) \| = \| \Delta(v) + \Delta( \mu_1\gamma_1\chi_{_{B_{11}}} + \overline{\mu_1}\gamma_1\chi_{_{B_{12}}} ) \| $$ $$= \| \Delta(v)- \Delta(- \mu_1\gamma_1\chi_{_{B_{11}}} - \overline{\mu_1}\gamma_1\chi_{_{B_{12}}} ) \| = \| v + \mu_1\gamma_1\chi_{_{B_{11}}} + \overline{\mu_1}\gamma_1\chi_{_{B_{12}}}  \| = 1.$$ This proves that the set  $\{ \Delta(v), \Delta(\gamma_1\chi_{_{B_{1}}}) \}$ is completely $M$-orthogonal, and consequently $\Delta(v)+ \Delta(\gamma_1\chi_{_{B_{1}}})\in S(X)$. Then there exists $b\in S(C(K))$ satisfying $\Delta(b)= \Delta(v)+ \Delta(\gamma_1\chi_{_{B_{1}}})$.\smallskip

We shall next show that $b=b\chi_{A_1\cup\cdots\cup A_m\cup B_1}$.  To this end, take an arbitrary $t_0\in K \backslash (A_1\cup\cdots\cup A_m\cup B_1)$, $\alpha\in\mathbb{T}$ and $\varphi\in\hbox{supp}(t_0,\alpha)$. By Proposition \ref{p kernel of the support} we have $\varphi\Delta(b)= \varphi\Delta(v)+ \varphi\Delta(\gamma_1\chi_{_{B_{1}}})=0$. Proposition \ref{p property of the support} gives $b=b\chi_{A_1\cup\cdots\cup A_m\cup B_1}$.\smallskip

%We can next mimic the arguments in the proof of Proposition \ref{p characteristic of clopen sets 1} to deduce that $|b(t)|=1$ for all $t\in A_1\cup\cdots\cup A_m\cup B_1$.\smallskip

Now, let us pick $t_0\in A_{k_0}$ for some $k_0\in\{1,\ldots,m \}$, and $\varphi\in \hbox{supp}(t_0, \lambda_{k_0})$. Proposition \ref{p kernel of the support} implies that $\varphi\Delta(b)= \varphi(\Delta(v)+ \Delta(\gamma_1\chi_{_{B_{1}}}))=1,$ and hence $\Delta(b)\in \varphi^{-1} (\{1\})\cap \mathcal{B}_{X} = \Delta(A(t_0,\lambda_{k_0}))$ (cf. Lemma \ref{l existence of support functionals for the image of a face}). Thus $b\in A(t_0,\lambda_{k_0})$, and it follows that $b(t_0)=\lambda_{k_0},$ for every $t_0\in A_{k_0}$. We conclude from the arbitrariness of $k_0$ that $b= \lambda_1\chi_{_{A_{1}}}+ \cdots + \lambda_m\chi_{_{A_{m}}}+ \gamma_1\chi_{_{B_{1}}} = v+\gamma_1\chi_{_{B_{1}}},$ which concludes the proof of the case $n=1$ in our induction argument.\smallskip
	
Suppose by the induction hypothesis that the statement is true for any $1\leq k\leq n$. By the induction hypothesis for $k=1$ and $k=n$ with the algebraic partial isometry $w = v + \gamma_1\chi_{_{B_{1}}},$ we get \begin{equation}\label{eq new 2602} \Delta(v)+ \sum_{j=1}^{n+1}\Delta(\gamma_j\chi_{_{B_{j}}})= \Delta(v)+ \Delta(\gamma_1\chi_{_{B_{1}}})+ \sum_{j=2}^{n+1}\Delta(\gamma_j\chi_{_{B_{j}}})
\end{equation}$$= \Delta(v + \gamma_1\chi_{_{B_{1}}})+ \sum_{j=2}^{n+1}\Delta(\gamma_j\chi_{_{B_{j}}})= \Delta\left(w + \sum_{j=2}^{n+1}\gamma_j\chi_{_{B_{j}}}\right)= \Delta\left(v + \sum_{j=1}^{n+1}\gamma_j\chi_{_{B_{j}}}\right) .$$

Let us take $\mu_1,\ldots,\mu_{n+1}\in \mathbb{T}$. For each $j\in \{1,\ldots,n+1\}$, Proposition \ref{p characteristic of clopen sets 2} assures the existence of two disjoint clopen sets $B_{j_1}$ and $B_{j_2}$ such that $B_j= B_{j_1}\cup B_{j_2}$ and $\mu_j\Delta(\gamma_j\chi_{_{B_{j}}})=\Delta( \mu_j\gamma_j\chi_{_{B_{j_1}}} + \overline{\mu_j}\gamma_j\chi_{_{B_{j_2}}} )$. Therefore, we can conclude by the identity proved in \eqref{eq new 2602}, applied twice to $v$ and $\{\mu_j\gamma_j\chi_{_{B_{j_1}}}\}_j$ and to $w=v + \sum_{j=1}^{n+1}\mu_j\gamma_j\chi_{_{B_{j_1}}}$ and $\{ \overline{\mu_j}\gamma_j\chi_{_{B_{j_2}}}\}_j$, and Proposition \ref{p characteristic of clopen sets 3 for two disjoint} that $$\left\| \Delta(v) + \sum_{j=1}^{n+1}\mu_j\Delta(\gamma_j\chi_{_{B_{j}}}) \right\| = \left\| \Delta(v) + \sum_{j=1}^{n+1}\Delta\left( \mu_j\gamma_j\chi_{_{B_{j_1}}} + \overline{\mu_j}\gamma_j\chi_{_{B_{j_2}}} \right) \right\| $$ $$= \left\| \Delta(v) + \sum_{j=1}^{n+1}\Delta( \mu_j\gamma_j\chi_{_{B_{j_1}}}) + \sum_{j=1}^{n+1}\Delta( \overline{\mu_j}\gamma_j\chi_{_{B_{j_2}}} ) \right\| $$ $$= \left\| \Delta\left(v + \sum_{j=1}^{n+1}\mu_j\gamma_j\chi_{_{B_{j_1}}}\right) + \sum_{j=1}^{n+1}\Delta( \overline{\mu_j}\gamma_j\chi_{_{B_{j_2}}} ) \right\| $$ $$= \left\| \Delta\left(v + \sum_{j=1}^{n+1}\mu_j\gamma_j\chi_{_{B_{j_1}}} + \sum_{j=1}^{n+1} \overline{\mu_j}\gamma_j\chi_{_{B_{j_2}}} \right) \right\| = 1,$$ which finishes the induction argument and the proof.
\end{proof}

Our next result is the technical core of the paper. In the statement we keep the notation given by Corollary \ref{c characteristic of clopen sets 2 a} and \eqref{eq def product odot}.

\begin{proposition}\label{p image of addition 2}
Suppose $K$ is a Stonean space, $\gamma_1,\ldots,\gamma_n\in \mathbb{T}$, $B_1,\ldots, B_n$ are non-empty disjoint clopen subsets of $K$ such that $B_1,\ldots,B_{j_0}$ are contained in $K_1$ and $B_{j_0+1},\ldots, B_n$ are contained in $K\backslash K_1$ with $j_0\in\{0,1,\ldots,n\}$. Suppose $X$ is a complex Banach space. Let $\Delta : S(C(K))\to S(X)$ be a surjective isometry and let $\displaystyle v=\sum_{k=1}^{m}\lambda_k\chi_{_{A_{k}}}$ be an algebraic partial isometry in $C(K),$ where $\lambda_1, \ldots, \lambda_m\in \mathbb{T}$, $A_1,\ldots,A_m$ are non-empty disjoint clopen sets in $K$ such that $A_k\cap B_j=\emptyset,$ for every $k\in\{1,\ldots,m\}$ and every $ j\in\{1,\ldots,n\}$. Then, given $\alpha_1,\ldots, \alpha_n\in\mathbb{C}\backslash\{0\}$ with $\max\{|\alpha_j| : j\in\{1,\ldots,n\}\}<1$, we have $$\Delta(v) + \sum_{j=1}^{n}\alpha_j\Delta(\gamma_j\chi_{_{B_{j}}}) = \Delta\left(v + \sum_{j=1}^{j_0}\alpha_j\gamma_j\chi_{_{B_{j}}} + \sum_{j=j_0}^{n}\overline{\alpha_j}\gamma_j\chi_{_{B_{j}}}\right)$$ $$= \Delta\left(v + \sum_{j=1}^{n}\alpha_j \odot (\gamma_j\chi_{_{B_{j}}})\right).$$
\end{proposition}

\begin{proof} Since the set $\{ \Delta(v), \Delta(\gamma_1\chi_{_{B_{1}}}), \ldots, \Delta(\gamma_n\chi_{_{B_{n}}}) \}$ is completely $M$-orthogonal (cf. Proposition \ref{p image of addition 1}), we can deduce that $\displaystyle \Delta(v)+ \sum_{j=1}^{n}\alpha_j \Delta(\gamma_j\chi_{_{B_{j}}})\in S(X)$. Thus there exists $y\in S(C(K))$ such that $\displaystyle\Delta(y)=\Delta(v)+ \sum_{j=1}^{n}\alpha_j \Delta(\gamma_j\chi_{_{B_{j}}})$.\smallskip

Let us fix $t_0\in K\backslash \left(\cup_{k,j} A_k\cup B_j\right)$, an arbitrary element $\mu$ of $\mathbb{T}$, and $\varphi\in\hbox{supp}(t_0,\mu)$. Proposition \ref{p kernel of the support} implies that $\displaystyle\varphi\Delta(y)=\varphi\Delta(v)+\sum_{j=1}^{n}\alpha_j \varphi\Delta(\gamma_j\chi_{_{B_{j}}})=0$. The arbitrariness of $\mu$ allows us to apply Proposition \ref{p property of the support} to deduce that $y(t_0)=0,$ which gives $y=y\chi_{_{\left(\cup_{k,j} A_k\cup B_j\right)}}$ thanks to the arbitrariness of $t_0$.\smallskip

Take now $t_0\in A_{k_0}$ and $\varphi\in\hbox{supp}(t_0,\lambda_{k_0})$ for some $k_0\in\{1,\ldots,m\}$. A new application of Proposition \ref{p kernel of the support} implies that $\displaystyle \varphi\Delta(y)=\varphi\Delta(v)+\sum_{j=1}^{n}\alpha_j \varphi\Delta(\gamma_j\chi_{_{B_{j}}})=1,$ and hence $\Delta(y)\in \varphi^{-1}(\{1\})\cap B_X=\Delta(A(t_0,\lambda_{k_0}))$, which assures that $y(t_0)=\lambda_{k_0}$, for every $t_0\in A_{k_0}$ and for every $k_0\in\{1,\ldots,m\}$. Therefore, $$y=v+y(1-\chi_{_{\cup_{k}A_k}})=v+\sum_{j=1}^{n}y\chi_{_{B_{j}}}.$$

We shall prove the desired identity by induction on $n$. If $n=1$, it follows from the above that there exists $y\in S(C(K))$ such that $\Delta(y)=\Delta(v)+ \alpha_1 \Delta(\gamma_1\chi_{_{B_{1}}})$ and $y=v+y\chi_{_{B_{1}}},$ with $|\alpha_1|<1$. We shall prove that $y\chi_{_{B_{1}}}=\alpha_1 \odot (\gamma_1\chi_{_{B_{1}}})$. The completely $M$-orthogonality of $\{\Delta(v), \Delta(\gamma_1\chi_{_{B_{1}}})\}$ guarantees the existence of $z\in S(C(K))$ such that $\Delta(z)=\Delta(v)+\frac{\alpha_1}{|\alpha_1|}\Delta(\gamma_1\chi_{_{B_{1}}})$ and since $B_1\subseteq K_1$ or $B_1\subseteq K_2$, the identity $z=v+\frac{\alpha_1}{|\alpha_1|} \odot (\gamma_1\chi_{_{B_{1}}})$ holds by Proposition \ref{p image of addition 1}, Corollary \ref{c characteristic of clopen sets 2 a} and \eqref{eq def product odot}. We also know that
$$
1-|\alpha_1|=\left|\frac{\alpha_1}{|\alpha_1|}-\alpha_1\right|=\| \Delta(y)-\Delta(z) \| = \|y-z\|= \left\| y\chi_{_{B_{1}}}-\frac{\alpha_1}{|\alpha_1|}\odot (\gamma_1\chi_{_{B_{1}}}) \right\|,
$$
and
\begin{equation}\label{15022018 1}1<
1+|\alpha_1|=\left|\frac{\alpha_1}{|\alpha_1|}+\alpha_1\right|=\left\| \Delta(y)
+\frac{\alpha_1}{|\alpha_1|}\Delta(\gamma_1\chi_{_{B_{1}}}) \right\|=\left\| y + \frac{\alpha_1}{|\alpha_1|}\odot (\gamma_1 \chi_{_{B_{1}}})\right\|
\end{equation} $$= \left\| y\chi_{_{B_{1}}}
+\frac{\alpha_1}{|\alpha_1|}\odot (\gamma_1 \chi_{_{B_{1}}})\right\|\vee \|v\|= \left\| y\chi_{_{B_{1}}}
+\frac{\alpha_1}{|\alpha_1|}\odot (\gamma_1 \chi_{_{B_{1}}})\right\|.$$

It follows from the previous two identities that $$\left| y(t)\chi_{_{B_{1}}}(t)-\frac{\alpha_1\gamma_1}{|\alpha_1\gamma_1|}\chi_{_{B_{1}}}(t) \right|\leq 1-|\alpha_1\gamma_1|,$$ and $$\left| y(t)\chi_{_{B_{1}}}(t)
+\frac{\alpha_1\gamma_1}{|\alpha_1\gamma_1|}\chi_{_{B_{1}}}(t) \right|\leq 1+|\alpha_1\gamma_1|,$$ for every element $t$ in $K_1$, and $$\left| y(t)\chi_{_{B_{1}}}(t)-\frac{\overline{\alpha_1}\gamma_1}{|\overline{\alpha_1}\gamma_1|}\chi_{_{B_{1}}}(t) \right|\leq 1-|\overline{\alpha_1}\gamma_1|,$$ and $$\left| y(t)\chi_{_{B_{1}}}(t)
+\frac{\overline{\alpha_1}\gamma_1}{|\overline{\alpha_1}\gamma_1|}\chi_{_{B_{1}}}(t) \right|\leq 1+|\overline{\alpha_1}\gamma_1|,$$ for every element $t$ in $K_2$. When particularized to an element $t\in B_1\subseteq K_1$ and $t\in B_1\subseteq K_1$ the previous inequalities result in $$\left| y(t)-\frac{\alpha_1\gamma_1}{|\alpha_1\gamma_1|} \right|\leq 1-|\alpha_1\gamma_1|,\hbox{ and } \left| y(t)
+\frac{\alpha_1\gamma_1}{|\alpha_1\gamma_1|}  \right|\leq 1+|\alpha_1\gamma_1|,$$ and $$\left| y(t) -\frac{\overline{\alpha_1}\gamma_1}{|\overline{\alpha_1}\gamma_1|}  \right|\leq 1-|\overline{\alpha_1}\gamma_1|,\hbox{ and } \left| y(t)
+\frac{\overline{\alpha_1}\gamma_1}{|\overline{\alpha_1}\gamma_1|} \right|\leq 1+|\overline{\alpha_1}\gamma_1|,$$ respectively, which give $y(t)= {\alpha_1\gamma_1}$ and $y(t)= {\overline{\alpha_1}\gamma_1}$,  respectively. Therefore $y\chi_{_{B_{1}}}=\alpha_1\odot (\gamma_1\chi_{_{B_{1}}}),$ which concludes the induction argument in the case $n=1$.\smallskip

Suppose now, by the induction hypothesis, that given $\alpha_1,\ldots, \alpha_n\in\mathbb{C}\backslash\{0\}$ with $\max\{|\alpha_j| : j\in\{1,\ldots,n\}\}<1$, we have $$\Delta(v) + \sum_{j=1}^{n}\alpha_j\Delta(\gamma_j\chi_{_{B_{j}}}) = \Delta\left(v + \sum_{j=1}^{j_0}\alpha_j\gamma_j\chi_{_{B_{j}}} + \sum_{j=j_0}^{n}\overline{\alpha_j}\gamma_j\chi_{_{B_{j}}}\right)$$ $$= \Delta\left(v + \sum_{j=1}^{n}\alpha_j \odot (\gamma_j\chi_{_{B_{j}}})\right),$$ whenever $v$ is an algebraic partial isometry and $B_1,\ldots, B_{n}$ are as in the statement of the proposition.\smallskip

Let $\displaystyle v=\sum_{k=1}^{m}\lambda_k\chi_{_{A_{k}}}$ and $B_1,\ldots, B_{n+1}$ be as in the statement of the proposition. By the arguments exhibited at the beginning of the proof, we may assume the existence of $y\in S(C(K))$ such that $\displaystyle \Delta(y)=\Delta(v)+ \sum_{j=1}^{n+1}\alpha_j \Delta(\gamma_j\chi_{_{B_{j}}})$ and $\displaystyle y=v+\sum_{j=1}^{n+1}y\chi_{_{B_{j}}}$. To prove that $\displaystyle  y=v+ \sum_{j=1}^{n+1}\alpha_j\odot (\gamma_j\chi_{_{B_{j}}})$, it will suffice to show that  $y\chi_{_{B_{j}}}=\alpha_j \odot (\gamma_j\chi_{_{B_{j}}})$ for every $j=1,\ldots,n+1$.\smallskip

Let us fix $j_1\in \{1,\ldots,n+1\}$. Proposition \ref{p image of addition 1} assures the existence of $z\in S(C(K))$ such that $$\Delta(z)=\Delta(v)+ \frac{\alpha_{j_1}}{|\alpha_{j_1}|} \Delta(\gamma_1\chi_{_{B_{1}}})+\sum_{j\neq j_1} \alpha_j \Delta(\gamma_j\chi_{_{B_{j}}}).$$ Thus, by induction hypothesis, Corollary \ref{c characteristic of clopen sets 2 a} and Proposition \ref{p image of addition 1}, we conclude that $\displaystyle z=v+ \frac{\alpha_{j_1}}{|\alpha_1|}\odot ( \gamma_1\chi_{_{B_{1}}})+\sum_{j\neq j_1} \alpha_j \odot (\gamma_j\chi_{_{B_{j}}})$. Applying this identity we get
\begin{equation}\label{15022018 2} 1-|\alpha_{j_1}|=\left|\alpha_{j_1}-\frac{\alpha_{j_1}}{|\alpha_{j_1}|}\right|=\|\Delta(y)-\Delta(z)\|
=\|y-z\| \end{equation}
$$=\left\|y\chi_{_{B_{j_1}}}-\frac{\alpha_{j_1}}{|\alpha_{j_1}|}\odot (\gamma_{j_1} \chi_{_{B_{j_1}}})\right\| \vee \max\left\{\left\| y\chi_{_{B_{j}}}-\alpha_{j}\odot (\gamma_{j}\chi_{_{B_{j}}}) \right\|: j\neq j_1 \right\}.$$ Consequently  \begin{equation}\label{15022018 3}\left\|y\chi_{_{B_{j_1}}}-\frac{\alpha_{j_1}}{|\alpha_{j_1}|}\odot (\gamma_{j_1} \chi_{_{B_{j_1}}})\right\| \leq 1-|\alpha_{j_1}|.
\end{equation}

Arguing as in \eqref{15022018 1} we also get \begin{equation}\label{15022018 1 bis}
1+|\alpha_{j_1}|=\left|\frac{\alpha_{j_1}}{|\alpha_{j_1}|}+\alpha_{j_1}\right|=\left\| \Delta(y)
+\frac{\alpha_{j_1}}{|\alpha_{j_1}|}\Delta(\gamma_{j_1}\chi_{_{B_{j_1}}}) \right\|=\left\| y + \frac{\alpha_{j_1}}{|\alpha_{j_1}|}\odot (\gamma_{j_1} \chi_{_{B_{j_1}}})\right\|
\end{equation} $$= \left\| y\chi_{_{B_{j_1}}}
+\frac{\alpha_{j_1}}{|\alpha_{j_1}|}\odot (\gamma_{j_1} \chi_{_{B_{j_1}}})\right\|\vee \|v\|\vee \max\{\| y\chi_{_{B_{j}}} \|: j\neq j_1 \}$$ $$= \left\| y\chi_{_{B_{j_1}}}
+\frac{\alpha_{j_1}}{|\alpha_{j_1}|}\odot (\gamma_{j_1} \chi_{_{B_{j_1}}})\right\|.$$

Evaluating at an element $t_0\in B_{j_1}$ we deduce from \eqref{15022018 3} and \eqref{15022018 1 bis} that $$\left| y(t_0)-\frac{\alpha_{j_1}\gamma_{j_1}}{|\alpha_{j_1}\gamma_{j_1}|} \right|\leq 1-|\alpha_{j_1}\gamma_{j_1}|,\hbox{ and } \left| y(t_0)
+\frac{\alpha_{j_1}\gamma_{j_1}}{|\alpha_{j_1}\gamma_{j_1}|}  \right|\leq 1+|\alpha_{j_1}\gamma_{j_1}|,$$ if $t_0\in K_1$, and $$\left| y(t_0) -\frac{\overline{\alpha_{j_1}}\gamma_{j_1}}{|\overline{\alpha_{j_1}}\gamma_{j_1}|}  \right|\leq 1-|\overline{\alpha_{j_1}}\gamma_{j_1}|,\hbox{ and } \left| y(t_0)
+\frac{\overline{\alpha_{j_1}}\gamma_{j_1}}{|\overline{\alpha_{j_1}}\gamma_{j_1}|} \right|\leq 1+|\overline{\alpha_{j_1}}\gamma_{j_1}|,$$ if $t_0\in K_2$, inequalities which give $y(t_0)= {\alpha_{j_1}\gamma_{j_1}}$ if $t_0\in K_1$ and $y(t_0)= {\overline{\alpha_{j_1}}\gamma_{j_1}}$  if $t_0\in K_2$,  respectively. We have shown that  $y\chi_{_{B_{j}}}=\alpha_j \odot (\gamma_j\chi_{_{B_{j}}})$, which finishes the proof.
\end{proof}

The next corollary is a straightforward consequence of the previous proposition.

\begin{corollary}\label{c gross identity}  Let $\Delta : S(C(K))\to S(X)$ be a surjective isometry, where $K$ is a Stonean space and $X$ is a complex Banach space. Let $v_1,\ldots,v_n$ be mutually orthogonal algebraic partial isometries in $C(K)$. Then, given $\alpha_1,\ldots, \alpha_n\in\mathbb{C}\backslash\{0\}$ with $\max\{|\alpha_j| : j\in\{1,\ldots,n\}\}=1$, we have $$\sum_{j=1}^{n}\alpha_j \Delta(v_j) = \Delta\left(\sum_{j=1}^{n}\alpha_j \odot v_j\right).$$ $\hfill\Box$
\end{corollary}

Proposition \ref{p image of addition 2} and its revision in Corollary \ref{c gross identity} are the tools we need to get a first approach to our main result. In this first approach we follow the ideas in the proof of \cite[Theorem 1.1]{Pe2017} or in the line in \cite{Ta:8}.

\begin{theorem}\label{t Tingley CK} Let $\Delta : S(C(K))\to S(X)$ be a surjective isometry, where $K$ is a Stonean space and $X$ is a complex Banach space. Then there exist two disjoint clopen subsets $K_1$ and $K_2$ of $K$ such that $K = K_1 \cup K_2$
satisfying that if $K_1$ {\rm(}respectively, $K_2${\rm)} is non-empty, then there exist a closed subspace $X_1$ {\rm(}respectively, $X_2${\rm)} of $X$ and a complex linear {\rm(}respectively, conjugate linear{\rm)} surjective isometry $T_1 : C(K_1) \to X_1$ {\rm(}respectively, $T_2 : C(K_2) \to X_2${\rm)} such that $X = X_1\oplus^{\infty} X_2$, and $\Delta (a) = T_1(\pi_1(a)) +  T_2(\pi_2(a))$, for every $a\in S(C(K))$, where $\pi_j$ is the natural projection of $C(K)$ onto $C(K_j)$ given by $\pi_j (a) = a|_{K_j}$. In particular, $\Delta$ admits an extension to a surjective real linear isometry from $C(K)$ onto $X$.
\end{theorem}

\begin{proof} Let $K_1$ and $K_2$ be the clopen subsets given by Corollary \ref{c characteristic of clopen sets 2 a}. We can assume that $K_j\neq \emptyset,$ for every $j=1,2$. Otherwise, the arguments are even easier. Clearly, $C(K) = C(K_1)\oplus^{\infty} C(K_2).$\smallskip

We consider the homogeneous extensions $F_j:C(K_j)\to X$, defined by $F_j(0)=0$ and $F_j(a) = \|a\| \Delta( \frac{1}{\|a\|} a)$ for all $a\in C(K_j)\backslash\{0\}.$\smallskip

Let us fix two algebraic elements in $C(K_1)$ (respectively,  $C(K_2)$) of the form
$$\widehat{a} = \sum_{j=1}^{n}\alpha_j \odot v_j,\hbox{ and } \widehat{b} =  \sum_{j=1}^{n} \beta_j \odot v_j,$$ where $v_1,\ldots, v_n$ are mutually orthogonal non-zero algebraic partial isometries in $K_1$ (respectively, $K_2$), $\alpha_1,\ldots, \alpha_n,$ $\beta_1,\ldots, \beta_n\in\mathbb{C}\backslash\{0\}$ with $\max\{|\alpha_j| : j\in\{1,\ldots,n\}\}=\|\widehat{a} \|$, and $\max\{|\beta_j| : j\in\{1,\ldots,n\}\}=\|\widehat{b}\|$.

If $\widehat{a}+\widehat{b} = 0,$ with $\widehat{a}\neq 0$,  then Corollary \ref{c gross identity} assures that $$F_j(\widehat{a}) ={\|\widehat{a}\|} \Delta\left(\frac{\widehat{a}}{\|\widehat{a}\|}\right) = {\|\widehat{a}\|} \left(- \Delta\left(-\frac{\widehat{a}}{\|\widehat{a}\|}\right)\right) = - {\|\widehat{b}\|} \Delta\left(\frac{\widehat{b}}{\|\widehat{b}\|}\right) = - F(\widehat{b}),$$ and hence $F_j(\widehat{a}+\widehat{b}) = 0 = F_j(\widehat{a}) + F_j(\widehat{b})$, for every $j=1,2$.\smallskip

If $\widehat{a}+\widehat{b}\neq 0$, a new application of Corollary \ref{c gross identity} implies that $$F_j(\widehat{a}) = \|\widehat{a}\| \Delta \left(\frac{1}{\|\widehat{a}\|} \widehat{a}\right)= \|\widehat{a}\| \Delta   \left( \sum_{j=1}^{n}\frac{\alpha_j}{\|\widehat{a}\|}  \odot v_j \right) = \|\widehat{a}\| \left( \sum_{j=1}^{n} \frac{\alpha_j}{\|\widehat{a}\|} \  \Delta ( v_j) \right),$$
$$F_j(\widehat{b}) = \|\widehat{b}\| \Delta \left(\frac{1}{\|\widehat{b}\|} \widehat{b}\right)= \|\widehat{b}\| \Delta   \left( \sum_{j=1}^{n}\frac{\beta_j}{\|\widehat{b}\|}  \odot v_j \right) = \|\widehat{b}\| \left( \sum_{j=1}^{n} \frac{\beta_j}{\|\widehat{b}\|} \  \Delta ( v_j) \right),$$
$$F_j(\widehat{a}+\widehat{b}) = \|\widehat{a}+\widehat{b}\| \Delta \left(\frac{1}{\|\widehat{a}+\widehat{b}\|} (\widehat{a}+\widehat{b})\right) = \|\widehat{a}+\widehat{b}\| \Delta   \left( \sum_{j=1}^n \frac{\alpha_j+\beta_j}{\|\widehat{a}+\widehat{b}\|} \odot v_j \right)$$ $$=  \sum_{j=1}^k (\alpha_j+\beta_j) \Delta (v_j).$$ Therefore, $F_j(\widehat{a}) + F_j(\widehat{b}) = F_j(\widehat{a}) + F_j(\widehat{b})$, for every $j=1,2$.\smallskip

It is known that $F_j$ is a Lipschitz mapping for every $j=1,2$ (compare for example, the final part in the proof of \cite[Theorem 1.1]{Pe2017}).\smallskip

Now we observe that for every $a,b\in C(K_j)$ and $\varepsilon >0$ we can find a set $\{v_1,\ldots,v_k\}$ of mutually orthogonal non-zero algebraic partial isometries in $C(K_j)$ and $\alpha_1,\beta_1,\ldots,\alpha_n,\beta_n \in \mathbb{C}\backslash\{0\}$ such that $\displaystyle \left\|a - \widehat{a}_k \right\|<\varepsilon$ and $ \left\|b - \widehat{b}_k \right\|<\varepsilon,$ where $\displaystyle \widehat{a}_k = \sum_{j=1}^k \alpha_j\odot v_j$, and $\displaystyle \widehat{b}_k=\sum_{j=1}^k \beta_j\odot v_j$. Since, by the arguments in the first part of this proof, we know that $F_j(\widehat{a}_k + \widehat{b}_k ) = F_j(\widehat{a}_k) + F_j(\widehat{b}_k),$ and $F_j$ is a Lipschitz mapping, we deduce, from the arbitrariness of $\varepsilon>0$, that $F_j(a+ b) = F_j(a) + F_j(b)$, for all $a,b\in C(K_j)$.\smallskip

For $\alpha\in \mathbb{C}$ and a non-zero algebraic partial isometry $v\in C(K_j)$ we have $$ F_1 (\alpha v) = |\alpha| \Delta \left(\frac{\alpha}{|\alpha|} v\right) = \alpha \Delta (v) = \alpha F_1 (v),$$ if $v\in C(K_1),$ and $$ F_2 (\alpha v) = |\alpha| \Delta \left(\frac{\alpha}{|\alpha|} v\right) = \overline{\alpha} \Delta (v) =\overline{\alpha} F_2 (v),$$ if $v\in C(K_2)$ (compare Corollary \ref{c characteristic of clopen sets 2 a}). We can therefore conclude from the arguments in the previous paragraph that $F_1$ is complex linear and $F_2$ is conjugate linear. It is obvious from definitions that $F_1 (a_1) = \Delta(a_1)$ and $F_2 (a_2) = \Delta(a_2)$ for every $a_j\in S(C(K_j))$, $j=1,2$. In particular, $F_1$ and $F_2$ are isometries, and $X_j = F_j (C(K_j))$ is a closed subspace of $X$ for every $j=1,2$.\smallskip

Furthermore, every $a\in S(C(K))$ can be approximated in norm by an algebraic element of the form $$\widehat{a} = \sum_{j=1}^{n}\alpha_j \odot v_j +  \sum_{k=1}^{m} \beta_k \odot w_k= \sum_{j=1}^{n}\alpha_j  v_j +  \sum_{k=1}^{m} \beta_k \odot w_k,$$ where $v_1,\ldots, v_n$ and $w_1,\ldots, w_m$ are mutually orthogonal non-zero algebraic partial isometries in $C(K_1)$ and $C(K_2)$, respectively, $\alpha_1,\ldots, \alpha_n,$ $\beta_1,\ldots, \beta_m\in\mathbb{C}\backslash\{0\}$ with $\max\{|\alpha_j| : j\in\{1,\ldots,n\}\} \vee \max\{|\beta_k| : k\in\{1,\ldots,m\}\}=1$. It follows from previous arguments (essentially from Corollary \ref{c gross identity}) that $$\Delta(\widehat{a}) = \sum_{j=1}^{n}\alpha_j  \Delta(v_j) +  \sum_{k=1}^{m} \beta_k \Delta(w_k) = F_1 (\pi_1(\widehat{a})) + F_2 (\pi_2(\widehat{a})),$$ and by continuity $$\Delta(a) = F_1 (\pi_1(a)) + F_2 (\pi_2(a)),$$ for every $a\in S(C(K))$. Suppose $x\in X_1\cap X_2$ with $\|x\|=1$. By construction, there exist $a_1\in S(C(K_1))$ and $a_2\in S(C(K_2))$ satisfying $\Delta(a_1) = x = \Delta(a_2),$ and hence $a_1=a_2$, which is impossible because $C(K_1)\cap C(K_2)$. Therefore, $X_1\cap X_2=\{0\}$.\smallskip

We shall finally show that $X= X_1\oplus X_2$. Given $x\in X$, there exists $a= a_1+a_2$ in $C(K),$ with $a_j \in C(K_j),$ satisfying $$ x= \Delta(a) =F_1 (\pi_1({a})) + F_2 (\pi_2({a}))=  F_1 (a_1) + F_2 (a_2)\in X_1\oplus X_2.$$ The rest is clear.
\end{proof}

After presenting our first approach to obtain the final conclusion in the previous Theorem \ref{t Tingley CK}, we insert next a second approach which is closer to the arguments in \cite{Ding07}, \cite[Corollaries 5 to 7]{Liu2007}, and \cite{FangWang06}. This second approach conducts to a less conclusive result, we include it here for completeness and as a tribute to the pioneering works of G.G. Ding, R. Liu and X.N. Fang, J.H. Wang.\smallskip

We recall next a lemma taken from \cite{FangWang06}.

\begin{lemma}\label{l 2.1 FangWang}\cite[Lemma 2.1]{FangWang06} Let $X$ and $Y$ be real normed spaces. Suppose $\Delta : S(X) \to S(Y)$ is an onto isometry. If for any $x, y \in S(X)$, we have $$\| \Delta(y) -\lambda \Delta(x) \|\leq \| y-\lambda x \|,$$ for all $\lambda > 0$, then $\Delta$ can be extended to a surjective real linear isometry from $X$ onto $Y$. $\hfill\Box$
\end{lemma}

Let $\Delta : S(C(K))\to S(X)$ be a surjective isometry, where $K$ is a Stonean space and $X$ is a complex Banach space. Let $K_1$ and $K_2$ be the clopen subsets given by Corolary \ref{c characteristic of clopen sets 2 a}. We define a new mapping $\sigma: K\times C(K)\to \mathbb{C},$ given by $\sigma(t,a) = a(t),$ if $t\in K_1,$ and $\sigma(t,a) = \overline{a(t)},$ if $t\in K_2$. By a little abuse of notation, we write $\sigma (a(t)) := \sigma(t,a)$ ($(t,a)\in K\times C(K)$).\smallskip

Our next proposition is a generalization of \cite[Theorem 3.1]{FangWang06} for complex-valued functions.

\begin{proposition}\label{p evaluation sigma deltat} Let $\Delta : S(C(K))\to S(X)$ be a surjective isometry, where $K$ is a Stonean space and $X$ is a complex Banach space. Then for each $t_0\in K$ and each $\varphi\in \hbox{supp} (t_0,1)$ the identity $$\varphi\Delta(a) = \sigma(t_0,a)= \sigma(a(t_0)),$$ holds for every $a\in S(C(K))$.
\end{proposition}

\begin{proof} As in the proof of Theorem \ref{t Tingley CK}, every $a\in S(C(K))$ can be approximated in norm by an algebraic element of the form $$\widehat{a} = \sum_{j=1}^{n}\alpha_j \odot v_j +  \sum_{k=1}^{m} \beta_k \odot w_k= \sum_{j=1}^{n}\alpha_j  v_j +  \sum_{k=1}^{m} \beta_k \odot w_k,$$ where $v_1,\ldots, v_n$ and $w_1,\ldots, w_m$ are mutually orthogonal non-zero algebraic partial isometries in $C(K_1)$ and $C(K_2)$, respectively, $\alpha_1,\ldots, \alpha_n,$ $\beta_1,\ldots, \beta_m\in\mathbb{C}\backslash\{0\}$ with $\max\{|\alpha_j| : j\in\{1,\ldots,n\}\} \vee \max\{|\beta_k| : k\in\{1,\ldots,m\}\}=1$. Corollary \ref{c gross identity} implies that $$\Delta(\widehat{a}) = \sum_{j=1}^{n}\alpha_j  \Delta(v_j) +  \sum_{k=1}^{m} \beta_k \Delta(w_k).$$ It is easy to check that for $t_0\in K$ and $\varphi\in \hbox{supp} (t_0,1)$ we have $$\varphi\Delta(\widehat{a}) = \sum_{j=1}^{n}\alpha_j  \varphi \Delta(v_j) +  \sum_{k=1}^{m} \beta_k \varphi \Delta(w_k) =\sigma(t_0,\widehat{a}) = \sigma(\widehat{a}(t_0)).$$ We can easily deduce from the continuity of $\Delta$ and $\sigma$, and the norm density commented above, that $\varphi\Delta(a) = \sigma(t_0,a)= \sigma(a(t_0))$.
\end{proof}

\begin{proof}[Alternative proof to the final conclusion in Theorem \ref{t Tingley CK}.] In hypotheses of this theorem, let $\Delta : S(C(K))\to S(X)$ be a surjective isometry. By Proposition \ref{p evaluation sigma deltat}, for each $t_0\in K$ and each $\varphi\in \hbox{supp} (t_0,1)$ the identity $$\varphi\Delta(a) = \sigma(t_0,a)= \sigma(a(t_0)),$$ holds for every $a\in S(C(K))$, equivalently, $$\varphi (x) = \sigma(t_0,\Delta^{-1}(x))= \sigma(\Delta^{-1} (x) (t_0)),$$ for every $x\in S(X)$. Let us pick $x, y \in S(X)$, $\lambda > 0$ and $\varphi_t\in \hbox{supp} (t,1)$. Since $$\| \Delta^{-1} (y) -\lambda \Delta^{-1} (x) \| = \max_{t\in K} |\Delta^{-1} (y) (t) -\lambda \Delta^{-1} (x) (t) |$$ $$= \max_{t\in K_1} |\Delta^{-1} (y) (t) -\lambda \Delta^{-1} (x) (t) |\vee \max_{t\in K_2} |\Delta^{-1} (y) (t) -\lambda \Delta^{-1} (x) (t) |$$ $$= \max_{t\in K_1} |\sigma(\Delta^{-1} (y) (t)) -\lambda \sigma(\Delta^{-1} (x) (t)) |\vee \max_{t\in K_2} |\sigma(\Delta^{-1} (y) (t)) -\lambda \sigma(\Delta^{-1} (x) (t)) | $$ $$= \max_{t\in K_1} |\varphi_t(y) -\lambda \varphi_t (x) |\vee \max_{t\in K_2} |\varphi_t(y) -\lambda \varphi_t (x)  | \leq \| y-\lambda x \|,$$ we conclude from \ref{l 2.1 FangWang} (see \cite[Lemma 2.1]{FangWang06}) that $\Delta^{-1}: S(X) \to S(C(K))$ admits a unique extension to a surjective real isometry from $X$ to $C(K)$. The rest is clear.
\end{proof}

We have commented at the introduction that for any $\sigma$-finite measure space $(\Omega,\mu)$, the complex space, $L^{\infty} (\Omega,\mu),$ of all complex-valued measurable essentially bounded functions equipped with the essential supremum norm, is a commutative von Neumann algebra, and thus from the metric point of view of Functional Analysis, the commutative von Neumann algebra $L^{\infty} (\Omega,\mu)$ is (C$^*$-isomorphic) isometrically equivalent to some $C(K),$ where $K$ is a hyper-Stonean space. Consequently, the next result, which is an extension of a theorem due to D. Tan \cite{Ta:8} to complex-valued functions, is a corollary of our previous Theorem \ref{t Tingley CK}.

\begin{theorem}\label{t Tingley Tan} Let $(\Omega,\mu)$ be a $\sigma$-finite measure space, and let $X$ be a complex Banach space. Suppose $\Delta: S(L^{\infty} (\Omega,\mu))\to S(X)$ is a surjective isometry. Then there exists a surjective real linear isometry $T: L^{\infty} (\Omega,\mu)\to X$ whose restriction to $S(L^{\infty} (\Omega,\mu))$ is $\Delta$.  $\hfill\Box$
\end{theorem}

\begin{remark}\label{remark extreme points}{\rm The celebrated Mazur-Ulam theorem assures that every surjective isometry $F$ between two real normed spaces $X$ and $Y$ is an affine function. P. Mankiewicz established an amazing generalization of the Mazur-Ulam theorem by showing that every bijective isometry between convex sets in normed linear spaces with non-empty interiors, admits a unique extension to a bijective affine isometry between the corresponding spaces (see \cite[Theorem 5 and Remark 7]{Mank1972}). Tingley's problem asks if every surjective isometry between the unit spheres of two normed spaces admits an extension to a surjective real linear isometry between the spaces. Tingley's problem remains open for general Banach spaces. We have survey some positive solutions to Tingley's problem in the introduction. The reader could feel tempted to ask if the unit spheres can be replaced by a strictly smaller set. In some operator algebras the unit spheres have been successfully replaced by the spheres of positive operators (see \cite{MolTim2003,MolNag2012,Nagy2013,Nagy2017} and \cite{Per2017}).\smallskip

Let $\partial_e(\mathcal{B}_X)$ denote the set of all extreme points of the closed unit ball, $\mathcal{B}_X$, of a Banach space $X$. The set $\partial_e(\mathcal{B}_X)$ seems to be an appropriate candidate to replace the unit sphere of $X$. However, the answer under these weak conditions is not always positive. Consider, for example, the real Banach space $X=\RR \oplus_\infty \RR.$ It is easy to check that $\partial_e (\mathcal{B}_X)=\{ p_1=(1,1),p_2=(1,-1),p_3=(-1,1),p_4=(-1,-1) \}$, with $d(p_i,p_j)= \|p_i-p_j\|=2(1-\delta_{i,j}),$ for every $i,j\in\{1,\ldots,4\}$. We can establish a surjective isometry $\Delta:\partial_e (\mathcal{B}_X)\to \partial_e (\mathcal{B}_X) $ defined by
$$\Delta (p_1)= p_2,\ \Delta(p_2)=p_3, \ \Delta(p_3)=p_4, \hbox{ and } \Delta(p_4)=p_1.$$
If we could find an extension of $\Delta$ to a surjective real linear isometry $T: X\to X$, then there would exist a real matrix satisfying $
T = \left(
\begin{array}{rl}
a & b \\
c & d
\end{array}
\right).$ However, by assumptions $T(p_1)=p_2 \Rightarrow a+b=1$ and $T(p_4)=p_1 \Rightarrow -a-b=1,$ which is impossible.\smallskip

After exhibiting the previous counterexample, we provide a list of examples where the previous Tingley's problem for extreme points admits a positive answer. If $H$ and $K$ are Hilbert spaces, we know well that $\partial_e(\mathcal{B}_H) = S(H)$ and $\partial_e(\mathcal{B}_K)= S(K).$ So, in this setting the set of extreme points coincides with the whole unit sphere. G.G. Ding proves in \cite[Theorem 2.2]{Ding2002} that every surjective isometry $$\Delta: \partial_e(\mathcal{B}_H) = S(H)\to \partial_e(\mathcal{B}_K)= S(K)$$ admits an extension to a surjective real linear isometry from $H$ onto $K$.\smallskip

A similar example can be given in another context. Let $C_p(H)$ be the space of $p$-Schatten von Neumann operators on a complex Hilbert space $H$ equipped with its natural norm $\|a\|_p^p :=\hbox{tr}(|a|^p)$. It is known that $C_p(H)$ is uniformly convex (and hence strictly convex) for every $1<p<\infty$ {\rm(}compare the Clarkson-McCarthy inequalities \cite{McCarthy67}{\rm)}. In particular, $\partial_e(\mathcal{B}_{C_p(H)})= S(C_p(H)).$ A very recent theorem assures that for $2<p<\infty$, every surjective isometry $$\Delta: \partial_e(\mathcal{B}_{C_p(H)}) = S(C_p(H))\to \partial_e(\mathcal{B}_{C_p(H)})= S(C_p(H))$$ can be uniquely extended to a surjective real linear isometry on $C_p(H)$ (see \cite[Theorem 2.15]{FerJorPer2018}).\smallskip

We can also present an example of different nature. It is well known that in a finite von Neumann algebra $M$, the set of all extreme points of its closed unit ball is precisely the set $\mathcal{U}_{_M}$ of all unitary operators in $M$ (see \cite{ChoKijNak69,Mil1964,Sin67}). An outstanding theorem due to M. Hatori and L. Moln{\'a}r establishes that every surjective isometry between the unitary groups of two von Neumann algebras can be extended to a surjective real linear isometry between the corresponding von Neumann algebras (compare \cite[Corollary 3]{HatMol2014}). Consequently, if $N_1$ and $N_2$ are finite von Neumann algebras (we could consider $N_1=N_2= \mathbb{C}\oplus^{\infty}\mathbb{C}$ or $N_1= N_2= M_n(\mathbb{C})$, and many other examples), every surjective isometry $$\Delta: \partial_e(\mathcal{B}_{_{N_1}}) = \mathcal{U}_{_{N_1}}\to \partial_e(\mathcal{B}_{_{N_2}}) = \mathcal{U}_{_{N_2}}$$ can be uniquely extended to a surjective real linear isometry $T: N_1 \to N_2$.
}
\end{remark}

\medskip\medskip

\textbf{Acknowledgements} Authors partially supported by the Spanish Ministry of Economy and Competitiveness (MINECO) and European Regional Development Fund project no. MTM2014-58984-P and Junta de Andaluc\'{\i}a grant FQM375.

\end{document}